\documentclass{amsart}

\usepackage{amsmath, amssymb,amsthm,amscd}
\usepackage[pdftex]{graphicx}
\usepackage{float}
\usepackage{pb-diagram}
\usepackage{amscd}
\usepackage{array}
\usepackage{hyperref}
\usepackage{comment}
\usepackage{mathtools}

\theoremstyle{definition}
\newtheorem{defi}{Definition}[section]
\newtheorem{thm}[defi]{Theorem}
\newtheorem{thm*}{Theorem}
\newtheorem{prop}[defi]{Proposition}
\newtheorem{lem}[defi]{Lemma}
\newtheorem{cor}[defi]{Corollary}
\newtheorem{conj*}{Conjecture}
\newtheorem{rem}[defi]{Remark}
\newtheorem{ex}[defi]{Example}

\numberwithin{equation}{section}

\begin{document}
\title{Blown-up boundaries associated with ample cones of K3 surfaces}
\author{Taiki Takatsu}
\thanks{Department of Mathematics, Tokyo Institute of Technology, Tokyo 152-8551, Japan.
\\
e-mail: takatsu.t.aa@m.titech.ac.jp
}
%%%%%%%%%%%%%%%%%%%%%%%%%%%%%%%%%%%%%%%%%%%%%%%%%%%%%%%%%%%%%%%%%%%%%%%%%%%%abstract
\begin{abstract}
In this paper, we apply the Bestvina-Mess type formula for relatively hyperbolic groups,
which is established by Tomohiro Fukaya, to automorphism groups of K3 surfaces, and we show that the virtual cohomological dimension of automorphism groups of K3 surfaces determined by the covering dimension of the blown-up boundaries associated with their ample cones.
\end{abstract}
\maketitle
%%%%%%%%%%%%%%%%%%%%%%%%%%%%%%%%%%%%%%%%%%%%%%%%%%%%%%%%%%%%%%%%%%%%%%%%%%%%Introduction
\section{Introduction}

In this paper, 
we discuss the classification problem of automorphism groups of K3 surfaces over $\mathbb{C}$.

There are two approaches to automorphisms of K3 surfaces.
The first one is the classification under some finiteness conditions, such as 
the classification of finite subgroups, and that of possible finite order of automorphisms.
For example, a well-known and important result in this approach is Mukai's result \cite{Mukai1} Theorem (0.3), which says that symplectic finite subgroups of the automorphism groups of K3 surfaces are isomorphic
to subgroups of the Mathieu group $M_{23}$, one of the sporadic simple groups.

The other approach is the application of Borcherds' lattice theory (often called Borcherds' method) to compute a generating set of infinite automorphism groups. The problem of finding
a generating set of the automorphism groups of Kummer surfaces was first proposed by Klein.
This problem took more than 100 years 
to be solved by Kondo \cite{Kondo1} 
using  Borcherds' method.
Although this method is very useful and, in fact, the automorphism groups of some specific K3 surfaces have been calculated,
this method is not applicable to the general case.

Shigeru Mukai recently proposed at the occasion of the Algebra Symposium in Tokyo in 2018 the following conjecture on discrete group theory and automorphism groups, from which one may hopefully find a general framework for dealing 
with infinite automorphism groups of elliptic K3 surfaces.

\begin{conj*}
Let $X$ be an elliptic K\textup{3}, Enriques or Coble surface,
and let $\operatorname{Aut}(X)$ be the automorphism group of X.
Then, the virtual cohomological dimension $\operatorname{vcd}(\operatorname{Aut}(X))$
of $\operatorname{Aut}(X)$ is equal to the maximum of the Mordell-Weil rank induced by all elliptic surface structures on $X$.
\end{conj*}

\begin{rem}
The virtual cohomological dimension of $\operatorname{Aut}(X)$
is easily seen to be larger than the rank of the Mordell-Weil group of any elliptic fibration.
Hence the main point of Mukai's conjecture lies in finding the upper-bound of
$\operatorname{vcd}(\operatorname{Aut}(X))$ as stated therein.
\end{rem}

This conjecture could be a turning point from the conventional study
of automorphism groups of elliptic K3, Enriques and Coble surfaces
by computing their generators to that by computing their  cohomological invariants.
Our main result is a formula which determines the virtual cohomological dimension of the automorphism group of a K3 surface by the covering dimension of the blown-up boundary
associated with the ample cone of the K3 surface.

\begin{thm*}
If $\operatorname{Aut}_{s}(X)$
is non-elementary,
then $$\operatorname{vcd}(\operatorname{Aut}(X)) = \dim \partial A_{X}^{bl} + 1,$$
where $\partial A_{X}^{bl}$ is the blown-up boundary associated with the ample cone of $X$
(see Section 5.2 for the detailed discussion).
\end{thm*}

To show this result, we focus on the natural action of automorphism groups on hyperbolic spaces, which come from the action on the cohomology of degree 2. 
Then we can consider the symplectic automorphism group of a K3 surface as a discrete subgroup
of M\"{o}bius transformations of a hyperbolic space. Moreover, we can show the following theorem:

\begin{thm*}\label{Kikuta-Takatsu}
Let $X$ be a K3 surface,
and let $\operatorname{Aut}_{s}(X)$ be the symplectic automorphism group of $X$.
Then $\operatorname{Aut}_{s}(X)$ is geometrically finite.
\end{thm*}

In \cite{Bestvina-Mess}, Bestvina and Mess showed that 
the cohomological dimension of a torsion free hyperbolic group is determined by
the covering dimension of its Gromov boundary.
Recently, Tomohiro Fukaya \cite{Fukaya} generalized the formula for relatively hyperbolic groups satisfying a certain condition. 
Geometrically finite groups are typical examples of relative hyperbolic groups. 
We can immediately obtain Theorem 1 by applying Fukaya's result \cite[Corollary B]{Fukaya}.

On the other hand, the paper of Fukaya is still unpublished, and his theorem and his proof are stated using geometric group theory.
For readability to algebraic geometers,
this paper also contains the result and the proof of Fukaya, 
which are simplified and modified to fit in with our situation, 
in Section 5.2,
using the geometry of ample cones of K3 surfaces and hyperbolic geometry.

\begin{rem}
After the author finished the first draft of this paper, 
Tomohiro Fukaya informed the author that Theorem 2 
has been independently obtained and proved by Kohei Kikuta.
He showed  geometrical finiteness for the automorphism groups of irreducible symplectic varieties over a field of characteristic zero, and his proof is different in some points from our proof given in this paper. 
For more details, 
the reader is advised to refer to his forthcoming paper.
\end{rem}

An apllication of Theorem 1 is the case that the ample cones of K3 surfaces are
sphere packings.
\begin{thm*}\label{Sphere packings}
Let $X$ be an elliptic K3 surface.
If $\partial A_{X}$ is a connected sphere packing,
and
every elliptic divisor of $X$ corresponds to 
the tangent point of some two spheres associated with $(-2)$-curves,
then
$$
\operatorname{vcd}(\operatorname{Aut}(X))
=\max \{ \operatorname{rk}\operatorname{MW}(f)\}
=\rho(X)-3
,$$
where $f$ runs all elliptic fibrations of $X$.
\end{thm*}

Baragar \cite{Baragar1} constructed K3 surfaces whose ample cones are sphere packings.
Then we can compute the virtual cohomological dimension of their automorphism groups.
Note that these K3 surfaces provide examples that affirm Conjecture 1, and,
to the best of our knowledge, they are the first non-trivial examples in the sense that the verification of the conjecture does not rely on explicit calculation of the group structure of the automorphism groups.

\vspace{3mm}
{\textbf{Acknowledgments}}

First, the author expresses his gratitude to his supervisor Fumiharu Kato,
for providing the author the interesting topics and for many helpful discussions.
He is grateful to Shigeru Mukai for various useful advices on
the virtual cohomological dimension of automorphism groups of K3 surfaces.
He thanks Takefumi Nosaka,
Ryokichi Tanaka, and Kohei Kikuta
for valuable comments and careful reading this paper.
Finally, he thanks Tomohiro Fukaya for many useful comments 
on hyperbolic geometry and hyperbolic group theory.

%%%%%%%%%%%%%%%%%%%%%%%%%%%%%%%%%%%%%%%%%%%%%%%%%%%%%%%%%%%%%%%%%%%%%%%%%%%%Lattices
\section{Lattices and K3 surfaces}
In this section, we recall some definitions and facts about lattices and K3 surfaces.
\subsection{Lattices}

Let $L$ be a free abelian group of rank $r$, and let $\langle \hspace{2mm}, \hspace{2mm} \rangle : L \times L\rightarrow \mathbb{Z}$ be a symmetric bilinear form. 
A symmetric bilinear form $\langle \hspace{2mm}, \hspace{2mm} \rangle $ is said to be {\it{non-degenerate}} if $\langle x, y \rangle = 0$ for any $y \in L$ implies $x=0$.
We call a pair $(L, \langle \hspace{2mm}, \hspace{2mm} \rangle)$ of $L$ and a non-degenerate symmetric bilinear form a {\it{lattice}} of rank $r$.
For an even lattice $L$, a quotient group $L^{*}/L$ is denoted by $A_{L}$. 

%%%%%%%%%%%%%%%%%%%%%%%%%
Let $L$ be a unimodular even lattice,
and $S$ be a primitive sublattice of $L$.
The orthogonal complement of $S$, denoted by $T$,
is also a primitive sublattice of $L$.
Let $H$ be a quotient group $L/(S \oplus T)$.
We denote by $p_{S}$ (resp. $p_{T}$) the projection map from $A_{S}\oplus A_{T}$ to $A_{S}$ 
(resp. $A_{T}$).

%K3 surfaces and elliptic surfaces
\subsection{K3 surfaces}

\begin{thm}[{\cite[Theorem 4.5]{Kondo2}}]
For a K3 surface $X$, the second cohomology group $H^{2}(X,\mathbb{Z})$ with the cup product is a unimodular even lattice with the signature $(3,19)$. Then $H^{2}(X,\mathbb{Z})$ is isomorphic to $U^{\oplus 3}\oplus E_{8}^{\oplus 2}$,
where $U$ is the hyperbolic lattice of rank 2.
\end{thm}

\begin{defi}
Let $\omega_{X}$ be a nowhere vanishing holomorphic $2$-form of $X$.
The {\it{N\'{e}ron-Severi lattice}} $NS(X)$ is defined by
$$NS(X) \coloneqq \{ x\in H^{2}(X,\mathbb{Z}) \hspace{+,3em}: \hspace{+,3em}\langle x,\omega_{X} \rangle =0\}=H^{2}(X,\mathbb{Z})\cap H^{1,1}(X,\mathbb{R}).$$
The N\'{e}ron-Severi lattice $NS(X)$ is a lattice by the cup product of signature $(1, \rho(X)-1)$,
where $\rho(X)$ is the Picard number of $X$.
The {\it{transcendental lattice}} $T_{X}$ of $X$ is defined by $NS(X)^{\bot}$.
\end{defi}

\begin{defi}
Let $X$ be a K3 surface.
We define
$$\Delta (X) \coloneqq \lbrace \delta \in S_{X} : \delta {\textup{ is a }} (-2){\textup{-root }}\rbrace$$
and 
$$\Delta^{+}(X) \coloneqq \lbrace \delta \in \Delta (X) : \delta {\textup{ is effective}} \rbrace .$$

For $\delta \in \Delta^{+}(X)$, the {\it{reflection}} $s_{\delta}$ in $H^{1,1}(X, \mathbb{R})$ is defined by
$$s_{\delta}(x)=x+ \langle x,\delta \rangle \delta \hspace{2mm}for \hspace{2mm} x\in L.$$
The {\it{reflection group}} $W(X)$ is the subgroup of $O(H^{1,1}(X, \mathbb{R}))$ generated by 
the set of reflections $\{ s_{\delta} : \delta \in \Delta (X) \rbrace.$
The cone $P(X)= \lbrace x \in H^{1,1}(X, \mathbb{R}) : \langle x, x \rangle > 0 \rbrace$ has two connected components and the one containing the K\"{a}hler class is denoted by $P^{+}(X)$ and called the {\it{positive cone}}.
\end{defi}

\begin{defi}
The group $W(X)$ acts on the space $P^{+}(X)$ in a canonical way.
The {\it{K\"{a}hler cone}} of $X$ is the one of the fundamental domains is given by
$$ D(X) = \lbrace x \in P^{+}(X) :  \langle x, \delta \rangle > 0 
{\text{ for any }}\delta \in \Delta(X)^{+}
\rbrace$$
containing K\"{a}hler class of $X$.
\end{defi}

\begin{defi}
Let $X$ be an algebraic K3 surface,
The {\it{ample cone}} $A(X)$ is defined to be 
$$A(X) = \lbrace x \in NS(X) \otimes \mathbb{R} \cap P^{+}(X) : 
\langle x, \delta  \rangle > 0 
{\text{ for any }}\delta \in \Delta(X)^{+}\rbrace.$$
\end{defi}

Let $f :X \to \mathbb{P}^{1}$ be an elliptic fibration of a K3 surface $X$,
and let $j:J\rightarrow \mathbb{P}^{1}$ be the jacobian fibration of $f$. 
The sections of $j$ generate a group $MW(f)$
that is callded the Mordell-Weil group of the elliptic fibration 
$f :X \to \mathbb{P}^{1}$.
The Mordell-Weil group $MW(f)$ is a finitely generated abelian group. 
We denote by $r(f)$ the rank of $MW(f)$.

The following formula is called Shioda-Tate formula.
\begin{lem}{\cite[Formula (4.3.2)]{Cossec-Dolgachev-Liedtke-Kondo}}\label{9}
Let $f:X\to \mathbb{P}^{1}$ be an elliptic fibration of a K3 surface, and let $F_{v}=f^{-1}(t_{v}) \; (1\leq v \leq k)$ be the singular fibres of $f$.
We denote by $m_{v}$  the number of irreducible components of $F_{v}$.
Then the $r(f)$ is given by the following formula:
 	$$r(f)= \rho(X)-2-\sum_{v=1}^k (m_{v}-1).$$
\end{lem}

%%%%%%%%%%%%%%%%%%%%%%%%%%%%%%%%%%%%%%%%%%%%%%%%%%%%%%%%%%%%%%%%%%%%%%%%%%%%section K3 surfaces, elliptic fibration and period map
%%%%%%%%%%%%%%%%%%%%%%%%%%%%%%%%%%%%%%%subsection

%%%%%%%%%%% Hyperbolic geometry %%%%%%%%%%%%%%
\section{Hyperbolic geometry}\label{chapter 4}
In this section, we recall some basic facts of $n$-dimensional hyperbolic spaces and discrete groups of isometries of $n$-dimensional hyperbolic spaces, which are contained in \cite{Ratcliffe}.

\vspace{5mm}
\noindent
{\textbf{Notation}}

Let $(X, d)$ be a metric space, and let $r$ be a positive real number. 
Throughout this section, we denote by 

$B(a, r)$  the open ball of radius $r$ centered at a point $a$ of $X$,

$C(a, r)$  the closed ball of radius $r$ centered at a point $a$ of $X$,

$S(a, r)$  the sphere of radius $r$ centerd at a point $a$ of $X$,

$N(S, r)$  the $r$-neighborhood of $S$ in $X$,

$I(X)$  the group of isometries of $X$.

We denote by 
$S^{n}$ the $n$-dimensional unit sphere.
A vector $e_{i}$ in $\mathbb{R}^{n}$ is the vector whose $i$-th entry is 1 
and the other entries are $0$.

In the hyperbolic geometry, it is useful to use some models of hyperbolic spaces for different situations.
Therefore, we begin with recalling 
the {\it{hyperboloid model}} of hyperbolic spaces.
In section \ref{the conformal ball model} and \ref{the upper half-space model}
we recall 
the {\it{conformal ball model}},
and the {\it{upper half-plane model}}
of hyperbolic spaces.

\subsection{The hyperboloid model}

\begin{defi}
Let $n$ be a non-negative integer.
Let $x=(x_{1}, \ldots, x_{n}, x_{n+1})$ and 
$y=(y_{1}, \ldots, y_{n}, y_{n+1})$ be vectors in the $(n+1)$-dimensional real vector space $\mathbb{R}^{n+1}$.
The {\it{Lorentzian inner product}} of $x$ and $y$ is defined by
$$ x\circ y = x_{1}y_{1}+ \cdots +x_{n}y_{n} -x_{n+1}y_{n+1}.$$
The {\it{Lorentzian (n+1)-space}} is the vector space $\mathbb{R}^{n+1}$ together with the Lorentzian inner product.
A real $(n+1)\times (n+1)$ matrix $A$ is said to be {\it{Lorentzian}} if
for any $x, y\in \mathbb{R}^{n+1}$, $x \circ y = Ax \circ Ay.$
 
\end{defi}

\begin{defi}
The {\it{hyperboloid model}} $H^{n}$ of hyperbolic $n$-spaces is defined to be the set 
$\{ x \in \mathbb{R}^{n+1} : x\circ x = -1,\hspace{1mm} x_{1}>0 \}.$
The {\it{hyperbolic distance}} $d_{H}(x, y)$ between $x$ and $y$ is defined by the equation
$$ \cosh d_{H}(x, y) = -x\circ y.$$
\end{defi}
By \cite[Theorem 3.2.2]{Ratcliffe}, the hyperbolic  distance $d_{H}$ is a metric on $H^{n}$.
By the definition of $d_{H}$,
Lorentz $(n+1)\times (n+1)$ matrices naturally act isometrically on $H^{n}$.
We denote the group of isometries by $I(H^{n}).$

\begin{defi}
Let $H^{n}$ be the hyperboloid model of $n$-dimensional hyperbolic spaces.
The set of all Lorentz $(n+1)\times (n+1)$ matrices with matrix multiplication
is called the {\it{Lorentz group}} of $(n+1) \times (n+1)$ matrices,
and is denoted by $O(n, 1)$.
The {\it{positive Lorentz group}} $O^{+}(n, 1)$ is defined by
$$ O^{+}(n, 1) = \{ A \in O(n, 1) : AH^{n} = H^{n} \}.$$
\end{defi}

By \cite[Theorem 3.2.3]{Ratcliffe}, there is an isomorphism 
from the positive Lorentzian group $O^{+}(n,1)$ to the group of hyperbolic isometries $I(H^{n})$.

\begin{defi}
A {\it{hyperbolic $m$-plane}} of $H^{n}$ is 
the nonempty intersection of $H^{n}$ with an $(m+1)$-dimensional
vector subspace of $\mathbb{R}^{n+1}$.
\end{defi}

%%%%%%%%%%%%%%%%%%%%%%%%%%%%%%%%%%%%%%%%%%%%%%%%%%%%%%%%%%%%%%%%%%%%%%%%%%%%%%%%%%%%%%%%%%%%
\subsection{M\"{o}bius transformations}
To define the conformal ball model and the upper half-space model of hyperbolic $n$-spaces,
we use the hyperboloid model and {\it{M\"{o}bius transformation}}.
In this section, we recall the definition and some basic facts of M\"{o}bius transformation.

\begin{defi}
Let $a$ be a unit vector in $E^{n}$, and let $t$ be a positive real number.
Consider the hyperplane of $E^{n}$ is defined by 
$$ P(a, t) = \{x \in E^{n} : a\cdot x = t\}.$$
The {\it{reflection}} $\rho$ of $E^{n}$ in the plane $P(a, t)$ is defined by 
$$ \rho(x) = x + 2(t-a\cdot x)a.$$

Let $a$ be a point in $E^{n}$, and let $r$ be a positive real number.
The {\it{sphere}} of $E^{n}$ of radius $r$ centered at $a$ is defined by 
$$ S(a, r) = \{x \in E^{n}: |x-a| = r\}.$$
The {\it{reflection}} (or {\it{inversion}}) $\sigma$ of $E^{n}$ in the sphere $S(a, r)$ is defined
by $$\sigma (x) = a + \bigg(\frac{r}{|x-a|}\bigg)^{2} (x-a).$$
\end{defi}

Let $\infty$ be a point not in $E^{n}$.
We define $\hat{E}^{n}=E^{n}\cup \{\infty\}$
whose topology is given by the one-point compactification of $E^n$.
\begin{defi}
Let $a$ be a unit vector in $E^{n}$, and let $t$ be a positive real number.
The {\it{extended hyperplane}} of $\hat{E}^{n}$ is defined by 
$\hat{P}(a, t) = P(a, t)\cup \{\infty\}.$
Let $\rho$ be the reflection of the hyperplane $P(a, t)$.
Extend $\rho$ to a bijection 
$\hat{\rho}: \hat{E}^{n} \rightarrow \hat{E}^{n}$ by setting
$\hat{\rho}(\infty) = \infty.$
In a similar way, the inversion $\sigma$ of an Euclidean sphere $S(a, r)$ in $E^{n}$ is extended to
$\hat{\sigma}:\hat{E}^{n} \rightarrow \hat{E}^{n}$ by setting $\hat{\sigma}(\infty) = \infty.$
The extended map $\hat{\rho}$ (resp. $\hat{\sigma}$) is called the reflection of 
$\hat{P}(a, t)$ (resp. $\hat{S}(a, r)$)
, which is a homeomorphism of $\hat{E}^{n}$.
\end{defi}

\begin{defi}
A {\it{sphere}} of $\hat{E}^{n}$ is either an extended hyperplane or an Euclidean sphere.
A {\it{M\"{o}bius transformation}} of $\hat{E}^{n}$ is a finite composition of reflection of a sphere in $\hat{E}^{n}$.
\end{defi}

%%%%%%%%%%%%%%%%%%%%%%%%%%%%%%%%%%%%%%%%%%%%%%%%%%%%%%%%%%%%%%%%%%%%%%%%%%%%%%%%%%%%%%%%%%%%
\subsection{The conformal ball model}\label{the conformal ball model}

In this section, we recall the conformal ball model of hyperbolic space.
Identify $\mathbb{R}^{n}$ with $\mathbb{R}^{n} \times \{0\}$ in $\mathbb{R}^{n+1}$.
Let $B^{n}$ be the open unit ball of $\mathbb{R}^{n},$
and let $|x|$ be the {\it{Euclidean norm}} of a vector $x$ in $\mathbb{R}^{n}$:
i.e., $|x|= (x_{1}^{2}+\cdots +x_{n}^{2})^{\frac{1}{2}}.$

The {\it{ stereographic projection $\zeta$}} from $B^{n}$ to the hyperbolic space $H^{n}$ is defined by 
the formula
$$\zeta (x) = 
\Bigl( \hspace{0.5mm}
\frac{2x_{1}}{1-|x|^{2}}, \ldots, \frac{2x_{n}}{1-|x|^{2}},\frac{1+|x|^{2}}{1-|x|^{2}}
\hspace{0.5mm}\Bigl).$$
This map $\zeta$ is bijective.

\begin{defi}
A metric $d_{B}$ on $B^{n}$ is defined by the formula $d_{B}(x, y)= d_{H}(\zeta (x), \zeta (y)).$
The metric space $(B^{n}, d_{B})$ is called the conformal ball model of hyperbolic $n$-spaces.
\end{defi}

\begin{defi}
A {\it{M\"{o}bius transformation of $B^{n}$}} is a M\"{o}bius transformation $\phi$ of $\hat{E}^{n}$
that leaves $B^{n}$ invariant: i.e. $\phi$ satisfies $\phi (B^{n}) = B^{n}.$
The group of M\"{o}bius transformation of $B^{n}$ is denoted by $M(B^{n})$.
\end{defi}
It is well-known that every M\"{o}bius transformation of $B^n$ restricts to an isometry
of the conformal ball model $B^n$ and the restriction induces an isomorphism from
$M(B^n)$ to $I(B^n).$

The boundary of $B^{n}$ in $\mathbb{R}^{n}$, denoted by $\partial B^{n}$, is 
the unit sphere $\{ x \in \mathbb{R}^{n} : |x|=1 \}$,
where $|x|$ is the Euclidean norm of $x$. A {\it{hyperbolic m-plane}} of $B^{n}$ is the image of a hyperbolic $m$-plane of $H^{n}$ by $\zeta^{-1}$.
By \cite[Theorem 4.5.3]{Ratcliffe}, a hyperbolic $m$-plane of $B^{n}$ is the intersection of $B^{n}$ with
either an $m$-dimensional vector subspace of $\mathbb{R}^{n}$ or an $m$-dimensional sphere of $\mathbb{R}^{n}$ orthogonal to $\partial B^{n}.$

\begin{defi}
Let $c$ be a point in $\partial B^{n}$.
Let $S$ be an $(n-1)$-dimensional inscribed sphere of $B^{n}$ tangent at $c$,
and let $T$ be the one of connected component of $\mathbb{R}^{n}\backslash S$ 
which is bounded.
We call the intersection of $B^{n}$ with $S$ (resp. $T$) 
the {\it{horosphere}} of $B^{n}$ (resp. {\it{horoball}} of $B^{n}$).
\end{defi}

%%%%%%%%%%%%%%%%%%%%%%%%%%%%%%%%%%%%%%%%%%%%%%%%%%%%%%%%%%%%%%%%%%%%%%%%%%%%%%%%%%%%%%%%%
\subsection{The upper half-space model}\label{the upper half-space model}
\begin{defi}
Let $x=(x_{1}, \ldots, x_{n})$ be a vector in $\mathbb{R}^{n}$,
and let $e_{n}$ be the vector in $\mathbb{R}^{n}$ whose entries are all zero except for the $n$-th,
which is one.
Let $\rho$ (resp. $\sigma$) be the reflection of $\hat{E}^{n}$ in 
$P(e_{n}, 0)$ (resp. $S(e_{n}, \sqrt{2})$).
The maps $\rho$ and $\sigma$ is given by the formula 
\begin{align*}
	\rho (x) &= (x_{1}, \ldots, -x_{n}),\\
	\sigma (x) &= e_{n} + \frac{2(x-e_{n})}{|x-e_{n}|^{2}} 
{\text{   for $x \in \mathbb{R}^{n}\backslash \{e_{n}\}$}}.
\end{align*}

The $n$-dimensional {\it{upper half-space}} $U^{n}$ is defined to be
$$ U^{n}=\{ (x_{1}, \ldots, x_{n}) \in \mathbb{R}^{n} : x_{n}>0 \}.$$
Then the map $\eta = \sigma \rho$ is a map 
from the upper half-space $U^{n}$ to the open unit ball $B^{n}$,
and is called the {\it{standard transformation}} from $U^{n}$ to $B^{n}.$
The {\it{Poincar\'{e} metric}} on $U^{n}$
is defined by
$$ d_{U}(x, y) = d_{B}(\eta (x), \eta (y)).$$

The pair $(U^{n}, d_{U})$ is a metric space, and is called the {\it{upper half-model}}
of hyperbolic $n$-spaces. By the definition, this is isomorphic to the conformal ball model $B^{n}$
of hyperbolic $n$-spaces.
The group of isometries of $U^{n}$ is denoted by $I(U^{n})$.
%The {\it{element of hyperbolic volume of the upper half-space model}} $U^{n}$ is given by
%$$ \frac{dx_{1}\cdots dx_{n}}{x_{n}^{n}}.$$
\end{defi}

\begin{defi}
A {\it{M\"{o}bius transformation of $U^{n}$}} is a M\"{o}bius transformation $\phi$
that leaves $U^{n}$ invariant: i.e. $\phi$ satisfies $\phi (U^{n}) = U^{n}.$
The group of M\"{o}bius transformation of $U^{n}$ is denoted by $M(U^{n}).$
\end{defi}

Every M\"{o}bius transformation of the upper half-space model $U^{n}$ restricts to
an isometry of $U^{n}$ and the 
restriction induces an isomorphism from $M(U^{n})$ to $I(U^{n}).$

\begin{defi}
Identify $E^{n-1}$ with $E^{n-1}\times \{0\}$ in $E^{n}.$
Let $\phi$ be a M\"{o}bius transformation of $\hat{E}^{n-1}$,
and let $\overline{a} = (a, 0)$.
The extension $\overline{\phi}$ of $\phi$ is defined as follows:
\begin{enumerate}
	\item If $\phi$ is the reflection of $\hat{E}^{n-1}$ in $P(a, t)$,
	then $\overline{\phi}$ is the reflection of $\hat{E}^{n}$ in $P(\overline{a}, t)$.
	
	\item If $\phi$ is the reflection of $\hat{E}^{n-1}$ in $S(a, t)$,
	then $\overline{\phi}$ is the reflection of $\hat{E}^{n}$ in $S(\overline{a}, t)$.
	
	\item In the general case, $\phi$ is the composition 
	$\phi= \sigma_{1}\cdots \sigma_{m}$ of reflections.
	Then, we define the map $\overline{\phi}$ by the composition $\overline{\phi} = \overline{\sigma_{1}} \cdots \overline{\sigma_{m}}.$
\end{enumerate}
We call the extension of $\phi$ the {\it{Poincar\'{e} extension}} 
\end{defi}

Poincar\'{e} extension is well-defined. Moreover, the following holds:
\begin{thm}
A M\"{o}bius transformation $\phi$ of $\hat{E}^{n}$ leaves $U^{n}$ invariant 
if and only if $\phi$ is the Poincar\'{e} extension of 
a M\"{o}bius transformation of $\hat{E}^{n-1}.$
In particular, Poincar\'{e} extension induces the isomorphism from $M(\hat{E}^{n-1})$
to $M(U^{n})$.
\end{thm}

\begin{defi}
Let $\eta$ be the standard transformation from $U^{n}$ to $B^n$.
A subset $P$ is a {\it{hyperbolic $m$-plane}} (resp. {\it{horosphere}}, {\it{horoball}}) of $U^n$
if the image $\eta (P)$ of $P$ by $\eta$ is a hyperbolic $m$-plane 
(resp. horosphere, horoball) in $B^{n}$.
\end{defi}

\begin{thm}[ {\cite[Theorem 4.6.3. and 4.6.5.]{Ratcliffe}} ]
\hspace{1mm}
\begin{enumerate}
	\item A subset $P$ of $U^n$ is a hyperbolic $m$-plane of $U^n$ if and only if
	$P$ is the intersection of $U^n$ with either an $m$-plane of $E^{n}$ orthogonal to $E^{n-1}$
	or an $m$-sphere of $E^{n}$ orthogonal to $E^{n-1}$.
	
	\item A subset $\Sigma$ of $U^n$ is a horopshere of $U^n$ 
	based at a point $b$ of $\hat{E}^{n-1}$ if and only if 
	$\Sigma$ is either a Euclidean hyperplane in $U^{n}$ parallel to $E^{n-1}$ if $b=\infty,$
	or the intersection with $U^n$ of a Euclidean sphere in $\overline{U}^n$ tangent to $E^{n-1}$
	at $b$ if $b\neq \infty.$

\end{enumerate}
\end{thm}

%%%%%%%%%%%%%%%%%%%%%%%%%%%%%%%%%%%%%%%%%%%%%%%%%%%%%%%%%%%%%%%%%%%%%%%%%%%%%%%%%%%%%%%%%%%%
\subsection{Types of transformations and elementary groups}

Let $B^{n}$ be the conformal ball model of $n$-dimensional hyperbolic spaces.
We write $\overline{B^{n}}$ for the closed unit ball $B^{n}\cup \partial B^{n}.$

\begin{defi}
Let $\phi$ be an element of $M(B^{n})$.
Then $\phi$ maps the closed unit ball $\overline{B^{n}}$ to itself.
By the Brouwer fixed point theorem, $\phi$ has a fixed point in $\overline{B^{n}}$.

\begin{enumerate}
\item $\phi$ is {\it{elliptic}} if $\phi$ has a fixed point in $B^{n}$,

\item $\phi$ is {\it{parabolic}} if $\phi$ has no fixed point in $B^{n}$ and has
a unique fixed point in $\partial B^{n}$,

\item $\phi$ is {\it{hyperbolic}} if $\phi$ has no fixed point in $B^{n}$ and fixes two
points in $\partial B^{n}$.

\end{enumerate}
\end{defi}

Let $\phi$ be an isometry of $U^{n}$, and let $\eta$ be the standard transformation from
$U^{n}$ to $B^{n}$. 
Then $\phi$ is said to be {\it{elliptic}} 
(resp. {\it{parabolic}}, {\it{hyperbolic}})
if $\eta \phi \eta^{-1}$ is elliptic (resp. parabolic, hyperbolic).

\begin{defi}
A subgroup $G$ of $M(B^{n})$ is {\it{elementary}} if
$G$ has a finite orbit in the closed ball $B^{n}$.
Elementary groups can also be classified into three types as follows:

\begin{enumerate}
	\item The group $G$ is of {\it{elliptic type}}
	if $G$ has a finite orbit in $B^{n}$.
	
	\item The group $G$ is of {\it{parabolic  type}}
	if $G$ fixes a point of $\partial B^{n}$ and has no other finite orbit in $\overline{B}^{n}$

	\item The group $G$ is of {\it{hyperbolic type}}
	if $G$ is neither of elliptic type nor parabolic type.

\end{enumerate}
\end{defi}

\begin{thm}[{  \cite[Theorem 5.5.2.]{Ratcliffe} }]
Let $\Gamma$ be a subgroup of $M(B^n)$.
Then $\Gamma$ is an elementary discrete subgroup of elliptic type
if and only if $\Gamma$ is conjugate in $M(B^{n})$ to 
a finite subgroup of the orthogonal group $O(n)$.

\end{thm}
\begin{thm}[{  \cite[Theorem 5.5.5]{Ratcliffe} }]
Let $\Gamma$ be a subgroup of $M(U^n).$
The group $\Gamma$ is an elementary discrete subgroup of parabolic type
if and only if $\Gamma$ conjugate in $M(U^n)$ to an infinite discrete subgroup
of $I(E^{n-1}).$
\end{thm}

\subsection{Dirichlet domains}
\begin{defi}
Let $X$ be a metric space.
A subset $D$ of $X$ is a {\it{fundamental domain}} for a group $\Gamma$ of isometries of $X$ if 

\begin{enumerate}
\item the set $D$ is open and connected in $X$,

\item the elements of $\{ gD : g\in \Gamma\}$ are mutually disjoint,

\item $\displaystyle X=\bigcup_{g\in \Gamma} \hspace{1mm}g\overline{D}$,

where $\overline{D}$ is the closure of $D$ in $X$.
\end{enumerate}
A fundamental domain $D$ is {\it{locally finite}} 
if and only if the set $\{g\overline{D}: g\in \Gamma\}$ is locally finite:
i.e.,
for each $x \in X$, there is an $\epsilon > 0$ such that
the open ball $B(x, \epsilon)$ of radius $\epsilon$ centered at $x$ meets only finitely many members of 
$\{g\overline{D}: g\in \Gamma\}$.
\end{defi}

Let $\Gamma$ be a discontinuous group of isometries of a hyperbolic $n$-space $H^{n}$.
By \cite[Theorem 6.6.10 and Theorem 6.6.12]{Ratcliffe},
there is a point $a$ such that the stabilizer group $\Gamma_{a}$ of $a$ is trivial.
For each $g\neq 1$ in $\Gamma$,
we set $H_{g}^{+}(a) = \{ x\in H^{n} : d_{H}(x, a) < d_{H}(x, ga) \}.$

\begin{defi}
The {\it{Dirichlet domain}} $D_{\Gamma}(a)$ for $\Gamma$, with center $a$,
is either $H^{n}$ if $\Gamma$ is trivial or
$$ D_{\Gamma}(a)=\bigcap_{g\in \Gamma\backslash \{1\}}H_{g}^{+}(a) .$$

\end{defi}

Let $D_{\Gamma}(a)$ be the Dirichlet domain for $\Gamma$ with center $a$ as above.
By \cite[Theorem 6.6.13]{Ratcliffe}, $D_{\Gamma}(a)$ is a locally finite fundamental domain for $\Gamma$.

%%%%%%%%%%%%%%%%%%%%%%%%%%%%%%%%%%%%%%%%%%%%%%%%%%%%%%%%%%%%%%%%%%%%%%%%%%%%%%%%%%%%%%%%%%%%
\subsection{Limit sets}

\begin{defi}
A point $a$ of $\partial B^{n}$ is a {\it{limit point}} of a subgroup $G$ of $M(B^{n})$
if there is a point $x\in B^{n}$ and  a sequence $\{g_{i} \}$ of elements of $ G$
such that the sequence $\{g_{i}x\}_{i\in \mathbb{N}}$ converges to $a$. 
The set of limit points of $G$ is called the {\it{limit set}} of $G$,
and is denoted by $\Lambda(G)$.
The complement of $\Lambda (G)$ in $\partial B^{n}$
is called the {\it{ordinary set}} of $G$,
and is denoted by $O(G).$
\end{defi}

\begin{defi}\label{bounded parabolic limit points}
A point $a$ of $\partial B^{n}$ is a {\it{bounded parabolic limit point}}
of a discrete group $\Gamma$ of $M(B^{n})$ if 
$a$ is a fixed point of a parabolic element of $\Gamma$ and the orbit space
$ (\Lambda(G)\backslash \{a\})/\Gamma$ is compact.
\end{defi}

%%%%%%%%%%%%%%%%%%%%%%%%%%%%%%%%%%%%%%%%%%%%%%%%%%%%%%%%%%%%%%%%%%%%%%%%%%%%%%%%%%%%%%%%%%%%
\subsection{Geometrically finite discrete groups}

\begin{defi}
Let $\Gamma$ be a discrete subgroups of $M(B^{n})$
with a parabolic element that fixes a point $a$ of $\partial B^{n}$.
Let $\Gamma_{a}$ be the stabilizer group of $a$.
A horoball $B(a)$ based at $a$ is said to be a {\it{horocusped region}} for $\Gamma$ based at $a$
if and only if for each $g\in \Gamma \backslash \Gamma_{a}$, 
we have $$ B(a) \cap gB(a) = \emptyset. $$
If a horocusped region $B(a)$ based at $a$ is not maximal with respect to the inclusion relation, $B(a)$ is said to be {\it{proper}}.
\end{defi}

\begin{defi}
Let $C(\Gamma)$ be the hyperbolic convex hull of the limit set $\Lambda(\Gamma)$ of $\Gamma$.
The {\it{convex core}} of $M=B^{n}/\Gamma$ is the set
$$C(M)= (C(\Gamma)\cap B^{n})/\Gamma.$$
\end{defi}

\begin{defi}\label{hyperbolic_geometry7}
Let $\Gamma$ be a discrete subgroup of $M(B^{n})$, let $M=B^{n}/\Gamma$,
and let $C(M)$ be the convex core of $M$.
The group $\Gamma$ is said to be {\it{geometrically finite}} if 
there is a (possibly empty) finite union $V$ of proper horocusps of $M$,
with disjoint closures,
such that $C(M)\backslash V$ is compact.
\end{defi}

\begin{thm}[{  \cite[Theorem 12.4.10]{Ratcliffe}  }]\label{hyperbolic_geometry6}
Let $\Gamma$ be a discrete subgroup of $M(B^{n})$
with an integer $n>1$.
Then the following are equivalent:

\begin{enumerate}
\item The orbit space $B^{n}/\Gamma$ has finite volume.

\item The group $\Gamma$ is geometrically finite and $\Lambda(\Gamma)=\partial B^{n}$.

\item There are finite points $p_{1}, \cdots p_{k}$ of $\overline{B^{n}}$ such that 
$P\cap B^{n}$ is a fundamental domain with finite volume,
where $P$ is the hyperbolic convex full of $\{p_{1}, \cdots, p_{k}\}$ in $\overline{B^{n}}.$

\end{enumerate}
\end{thm}

\begin{thm}[{  \cite[Theorem 12.6.5]{Ratcliffe} }]\label{Ratcliffe_thm_12.6.5}
Let $\Gamma$ be a subgroup of $M(B^n)$,
and let $C$ be the set of cusped limit points of $\Gamma$.
For each $c\in C$, there exists a cusped region $U(c)$ based at $c$ for $\Gamma$ 
such that 
\begin{enumerate}
	\item $\{\overline{U}(c) : c\in C\}$ is mutually disjoint,
	
	\item $ \gamma U(c) = U(\gamma c)$ for any $\gamma \in \Gamma$ and $c \in C$,
	
	\item $\{ \overline{U}(c)\backslash \{c\}\} $ is a locally finite family
	      of closed subsets of $B^{n}\cup O(\Gamma).$

\end{enumerate}

\end{thm}

%%%%%%%%%%% Cohomology of Groups %%%%%%%%%%%%%%
\section{Cohomological dimension of groups}
\subsection{Cohomology of groups}
In this section, we review some definitions and basic facts about cohomology of groups,
which are contained in \cite{Brown}.

%--------------------------definition of proj dim---------------
\begin{defi}\label{23}
Let $R$ be a ring and $M$ be a left $R$-module.
Let $n$ is a non-negative integer. 
Then $\operatorname{proj dim}_{R} M \leq n$
if there is a left projective resolution of length $n$:
$$ 0 \rightarrow P_{n} \rightarrow P_{n-1} \rightarrow \cdots \rightarrow P_{0} \rightarrow M \rightarrow 0.$$
\end{defi}
The following lemma is well-known.
%----------------------------------------------------------------
\begin{lem}[{\cite[Chapter VIII, Lemma (2.1)]{Brown}}]\label{24}
The following conditions are equivalent:
%----------------the conditions of proj dim---------------
\begin{enumerate}
\item $\operatorname{proj dim}_{R} M \leq n.$
 
\item $\operatorname{Ext}_{R}^{i} (M, -) = 0$ for $i>n$.

\item $\operatorname{Ext}_{R}^{n+1} (M, -) =0.$

\item For any left $R$-module exact sequence 
$0 \rightarrow K \rightarrow P_{n-1} \rightarrow \cdots \rightarrow P_{0}\rightarrow 0$ with $P_{i}$ projective,  $K$ is projective.
\end{enumerate}
\end{lem}

%-----------------Definition of the cohomological dimension of groups------------
\begin{defi}\label{25}
We define the{\it{ cohomological dimension}} $\operatorname{cd} \Gamma$ as following:
\begin{equation*}
\begin{split}
  \operatorname{cd} \Gamma &= \operatorname{projdim}_{\mathbb{Z} \Gamma} \mathbb{Z} \\
  &= \inf\lbrace n : \mathbb{Z} {\textup{ admits a projective resolution of lengh }}n \rbrace \\
  &= \inf\lbrace n : H^{n}(\Gamma, -) = 0 {\textup{ for }} \forall i >n \rbrace \\
  &= \sup\lbrace n : H^{n}(\Gamma, -) \neq 0 {\textup{ for some }} \Gamma{\textup{-module }} M\rbrace
\end{split}
\end{equation*}
Recall that the functor $\operatorname{Ext}_{R}^{i} (M, -)= H^{i} (\operatorname{Hom}_{R} (M,-))$.
Then $\operatorname{Ext}_{\mathbb{Z}\Gamma}^{i} (M, -) = H^{i}(\Gamma, -)$.
By the previous Lemma, the above definition is well-defined.
\end{defi}
%---------------------------------------------------------------------
The following theorem is called Serre's theorem 
that claims well-definedness of the {\it{virtual cohomological dimension}}.
%--------------------Serre's Theorem-----------------------------------
\begin{thm}[{\cite[Chapter V\hspace{-0.5mm}I\hspace{-0.5mm}I\hspace{-0.5mm}I, Theorem (3.1)]{Brown}}]\label{26}
If $\Gamma$ is a torsion-free group and $\Gamma^{'}$ is a subgroup of finite index,
then $\operatorname{cd} \Gamma = \operatorname{cd} \Gamma^{'}$.
\end{thm}

%--------------Definition of virtual cohomological dimension-------
\begin{defi}\label{27}
Let $\Gamma$ be a virtually torsion-free group.
By Theorem \ref{26}, any torsion-free subgroup of finite index has the same cohomological dimension.
Then the dimension is called the {\it{virtual cohomological dimension}},
and is denoted by $\operatorname{vcd} \Gamma$.
\end{defi}
%-----------------------------------------------------------------

\subsection{Bestvina-Mess formula}
In \cite{Bestvina-Mess}, Bestvina and Mess proved the following formula 
for any virtually torsion free hyperbolic group $G$:
$$ \operatorname{vcd}G = \dim \partial G + 1,$$
where $\partial G$ is the Gromov boundary of $G$ and $\dim \partial G$ is the {\it{covering dimension}} of $\partial G$.
In \cite{Bestvina-Mess}, this formula, called the Bestvina-Mess formula, is generalized for groups which admits a {\it{$\mathcal{Z}$-structure.}} Recently, Fukaya \cite{Fukaya} proved the Bestvina-mess formula for relatively hyperbolic groups.

\begin{defi}
A metrizable space $X$ is called an {\it{abusolute neighborhood retract (ANR)}}
(resp. an {\it{abusolute retract (AR)}}) if
$X$ is a neighborhood retract (resp. a retract) of an arbitrary metrizable space that
contains $X$ as a closed subspace.
\end{defi}
It is well-known that an ANR space $X$ is an AR if and only if $X$ is contractible.

\begin{defi}
Let $X$ be a ANR space.
A closed subset $Z$ of $X$ is called a $\mathcal{Z}$-set if $X\backslash Z$ is homotopy dense in $X$,
i. e.,
there exists a homotopy $H:X\times \lbrack 0, 1\rbrack :\rightarrow X$ such that 
for any $t>0$, $H_{t}(X) \subset A$.

\end{defi}

\begin{defi}
Let $\overline{X}$ be a compact AR space,
and let $Z$ be a closed subset of $\overline{X}$.
A pair $(\overline{X}, Z)$ is a $\mathcal{Z}$-structure for a group $G$
if

\begin{enumerate}
	\item $Z$ is a $\mathcal{Z}$-set of $\overline{X}$.
	
	\item There exists a proper metric $d$ on $X=\overline{X}\backslash Z$ such that
	
	2-1. $G$ geometrically acts on $(X,d).$
	
	2-2. For any compact subset $K$ of $X$ and any open covering $\mathcal{U}$ of $\overline{X}$,
	there exists $U_{g} \in \mathcal{U}$ 
	such that $gK \subset U_{g}$ for all $g \in G$ exept finitely many elements of $G$.
	
\end{enumerate}
\end{defi}

\begin{thm}[{ \cite[Theorem 1.7]{Bestvina}}]\label{Bestvina-Mess formula}
Let $(\overline{X}, Z)$ be a $\mathcal{Z}$-structure on
a virtually torsion free group $G$.
Then $\operatorname{vcd} G = \dim Z+1,$
where $\dim Z$ is the covering dimension of $Z$.
\end{thm}

\begin{ex}
Let $\mathbb{B}^{n}$ be the $n$-dimensional conformal model of hyperbolic spaces.
and let $\Gamma$ be a cocompact discrete subgroup of $M(B^n).$
Then the pair $(\overline{\mathbb{H}}^{n}, S^{n-1})$ is a $\mathcal{Z}$-structure of
a group $\Gamma$, and so $\operatorname{vcd} \Gamma = \dim S^{n-1} +1 = n.$
\end{ex}

%%%%%%%%%%%%%%%%%%%%%%%%%%%%%%%%%%%%%%%%%%%%%%%%%%%%%%%%%%%%%%%%%%%%%%%%%%%%
\section{Automorphism groups of K3 surfaces and virtual cohomological dimension}\label{main result}

In this section, 
we consider the natural actions of the automorphism groups of K3 surfaces
on hyperbolic spaces.
First, we recall the constructions of hyperbolic spaces induced by hyperbolic lattices,
and discuss the actions of the automorphism groups of K3 surfaces on certain hyperbolic spaces, 
and moreover some properties of automorphism groups.
Secondly, we  discuss some invariants of automorphism groups, for example the virtual cohomological dimension.

%%%%%%%%%% geometrically finiteness of Auts(X) %%%%%%%
\subsection{Geometrical finiteness of $\operatorname{Aut}_{s}(X)$}

Let $X$ be an algebraic K3 surface defined over $\mathbb{C}$,
and let $f$ be the quadratic form defined by 
$f(x) = \langle x, x \rangle $ for $x \in NS(X) \otimes \mathbb{R}$,
where $\langle, \rangle$ is the inner product on $NS(X)$ 
induced by the cup product on $H^{2}(X, \mathbb{Z})$.
$f$ has the signature $(1, n)$, where $n=\rho (X) -1.$
Then there exists $M \in \mathrm{GL}(n+1, \mathbb{R})$ such that $f_{n+1}(Mx) = f(x)$,
where $f_{n+1}(x)=x_{0}^{2} -(x_{1}^{2}+x_{2}^{2}+ \cdots +x_{n}^{2}).$
Therefore, we can identify the component of the set $$\{ x \in NS(X)\otimes \mathbb{R} : \langle x, x \rangle =1 \}$$
containing ample classes of $X$ with the hyperboloid model of hyperbolic $n$-space.
We denote this hyperboloid model of hyperbolic $n$-space by $\mathbb{H}^{n}_{X}$.
The hyperbolic distance function $d$ on $\mathbb{H}^{n}_{X}$
satisfies $\cosh d(x, y) = \langle x, y \rangle.$

%%%%%%%%%%%%%%%%%%%%%%%%%%%%%%%%%%%%%%%%%%%%%%%%%%%%%%%%%%%%%%%%%%%%%%%%%%%
\begin{defi}
We call $\mathbb{H}^{n}_{X}$ the {\it{hyperbolic space associated to $X$.}}
Let $\overline{P^{+}}(X)$ be the closure of the ample cone $P^{+}(X)$
in $NS(X)\otimes \mathbb{R}.$
We define $\mathbb{P}(\overline{P^{+}}(X))$
by 
$$\mathbb{P}(\overline{P^{+}}(X))
= (\overline{P^{+}}(X)\backslash \{0\})/\mathbb{R}_{> 0}.$$
For a point $x$ in $\overline{P^{+}}(X)\backslash\{0\}$,
we denote by $\lbrack x \rbrack$ the class in $\mathbb{P}(\overline{P^{+}}(X))$
corresponding to $x$.
The inclusion from $\mathbb{H}^{n}_{X}$ to $\overline{P^{+}}(X)$ induces
a continuous injective map 
$i :\mathbb{H}^{n}_{X}\rightarrow \mathbb{P}(\overline{P^{+}}(X)).$
Since $i$ is a homeomorphism from $\mathbb{H}^{n}_{X}$ to $i(\mathbb{H}^{n}_{X})$,
we identify $\mathbb{H}^{n}_{X}$ with $i(\mathbb{H}^{n}_{X}).$

The {\it{boundary of $\mathbb{H}^{n}_{X}$ at infinity}},
denoted by $\partial {\mathbb{H}}^{n}_{X}$,  
is the closure of 
$i(\mathbb{H}^{n}_{X})$ in $\mathbb{P}(\overline{P^{+}}(X)),$
which is equals to the set
$\{ x \in \overline{P^{+}}(X)\backslash\{0\}: \langle x, x \rangle =0\}/\mathbb{R}_{> 0}.$
We set $\overline{\mathbb{H}}^{n}_{X} = \mathbb{H}^{n}_{X} \cup \partial {\mathbb{H}}^{n}_{X}.$
\end{defi}

%%%%%%%%%%%%%%%%%%%%%%%%%%%%%%%%%%%%%%%%%%%%%%%%%%%%%%%%%%%%%%%%%%%%%%%%%%%

%%%%%%%%%%%%%%%%%%%%%%%%%%%%%%%%%%%%%%%%%%%%%%%%%%%%%%%%%%%%%%%%%%%%%%%%%%%
\begin{defi}
Let $A(X)$ be the ample cone.
We define $$\operatorname{Aut} (A(X)):= \lbrace \phi \in O(NS(X)): \phi A(X) = A(X) \rbrace
{\text{ and }}
A_{X} := A(X) \cap \mathbb{H}^{n}_{X}
$$

The closure of $A_{X}$ in $\overline{\mathbb{H}^{n}_{X}}$ is denoted by $\overline{A_{X}}.$
We denote $\overline{A_{X}}\backslash A_{X}$ by $\partial A_{X}$,
called {\it{the boundary associated with the ample cone of $X$}}.
\end{defi}
%%%%%%%%%%%%%%%%%%%%%%%%%%%%%%%%%%%%%%%%%%%%%%%%%%%%%%%%%%%%%%%%%%%%%%%%%%%

By the unimodulality of $H^{2}(X, \mathbb{Z})$,
the discriminant groups of $NS(X)$ and $T_{X}$ are isomorphic:
$$ A_{NS(X)} = NS(X)^{*}/NS(X) \simeq T_{X}^{*}/T_{X} = A_{T_{X}}.$$
Each isometry $\phi$ of $NS(X)$ or $T_{X}$ gives rise to an isometry of the group
$A_{NS(X)} \simeq  A_{T(X)}$, denoted by $\phi^{*}$.
%We have
%\begin{equation*}
%	\begin{split}
%	&O(  H^{2}(X, \mathbb{Z}) \cap (  O(NS(X)) \times O(T(X))   ) \\
%	&= \lbrace (\phi, \psi) \in O(NS(X)) \times O(T(X)) : \phi^{*} = \psi^{*} \rbrace.
%	\end{split}
%\end{equation*}
%
%By Torelli theorem,
%the automorphism group $\operatorname{Aut}(X)$ is described as
%$$\operatorname{Aut} (X) =
%\lbrace (\phi, \psi) \in \operatorname{Aut}(A(X)) \times U:\phi^{*} =\psi^{*} \rbrace,$$
%where $U$ is the subgroup of $O(T_{X})$ consisting of isometries that
%multiply holomorophic 2-form $\omega_{X}$ by a complex number.
Since the automorphism group $\operatorname{Aut}(X)$ of $X$ naturally acts on the ample cone of $X$,
there is a natural homomorphism $\rho : \operatorname{Aut} (X) \rightarrow \operatorname{Aut}(A(X)).$
As for this homomorphism,
the following theorem is well-known.
%%%%%%%%%%%%%%%%%%%%%%% Theorem
\begin{thm}[{\cite[Theorem 8.1]{Kondo2}}]\label{aut1}
The kernel and the cokernel of $\rho$ are finite groups.
\end{thm}

Since $\operatorname{Aut}(X)$ is finitely generated \cite[Proposition 2.2]{Sterk}
and Selberg's lemma \cite{Alperin}, $\operatorname{Aut}(X)$ is virtually torsion free.
\begin{defi}
Let $X$ be a K3 surface,
and let $\omega_{X}$ be a nowhere vanishing holomorphic 2-form of $X$.
A automorphism $\phi$ of $X$ is said to be {\it{symplectic}} 
if $\phi$ acts trivially on $\omega_{x}$.
The {\it{symplectic automorphism group}} of $X$, denoted by $\operatorname{Aut}_{s}(X)$,
is the group of symplectic automorphisms of $X$.

\end{defi}

The following lemma is well-known.
\begin{lem}[{\cite [Lemma 8.11]{Kondo2}}]\label{Kondo8.11}
The symplectic automorphism of $X$ acts trivially on the transcendental lattice $T_{X}$.
\end{lem}

Let $\phi$ be the symplectic automorphism of $X$.
Then, by Lemma \ref{Kondo8.11}, $\phi$ acts trivially on $A_{NS(X)}\simeq A_{T_{X}}.$
By the Torelli theorem, we can identify $\operatorname{Aut}_{s}(X)$ with the subgroup 
$\{ \phi \in O^{+}(NS(X)) : \phi A(X) = A(X) , \phi^{*}|_{A_{NS(X)}} = \mathrm{id}_{A_{NS(X)}} \}$.
Therefore, $\operatorname{Aut}_{s}(X)$ is a subgroup of finite index in $\operatorname{Aut}(A(X)).$
Since any element of $O^{+}(NS(X))$
is an isometry of the hyperbolic space $\mathbb{H}^{n}_{X}$
associated $X$, the symplectic automorphism group $\operatorname{Aut}_{s}(X)$ of $X$ is 
an $n$-dimensional discrete subgroup of isometries of $\mathbb{H}^{n}_{X}$.

%%%%%%%%%%%%%%%%%%%%%%% Prop
\begin{prop}\label{ample cone vs Dirichlet domain}
Let $X$ be a K3 surface, and let $a \in A_{X}.$
Then the Dirichlet domain $D(a)$ of $\mathbb{H}^{n}_{X}$, with center $a$, for the reflection group
$W(X)$ is equal to $A_{X}=A(X)\cap \mathbb{H}^{n}_{X}.$
\end{prop}
\begin{proof}
Recall that $H^{+}_{w}(a) = \{ x \in \mathbb{H}^{n} : d(x, a) < d(x, wa) \hspace{1mm}\}$
for $w \in W_{X}.$
If $w$ is the reflection $s_{\delta}$ for some $(-2)$-root $\delta$,
\begin{align*}
&\hspace{7mm} d(x, a)< d(x, s_{\delta}a) \\
&\Leftrightarrow \hspace{1mm} \cosh d(x, a) < \cosh d(x, s_{\delta}a) \\
&\Leftrightarrow \hspace{1mm} \langle x, a \rangle < \langle x, s_{\delta}a \rangle \\
&\Leftrightarrow \hspace{2mm}  0< \langle a, \delta \rangle \langle x, \delta \rangle \\
&\Leftrightarrow \hspace{2mm}  0< \langle x, \delta \rangle
\end{align*}
Then we have that 
$H^{+}_{s_{\delta}}(a) = \{x \in \mathbb{H}^{n}_{X}: \langle x, \delta \rangle >0 \}.$ 
Since $A(X)$ is a fundamental domain for $W(X)$ on the positive cone $P^{+}(X)$ of $X$ whose sides
are bounded by $(-2)$-root, we have 
$A_{X}=\cap_{\delta \in \Delta^{+}(X)} H^{+}_{s_{\delta}}.$

By the definition of $D(a)$, we have $D(a) \subset A_{X}.$
Since $D(a)$ and $A_{X}$ are fundamental domains for $W(X)$, we have $D(a)=A_{X}.$
\end{proof}

%%%%%%%%%%%%%%%%%%%%%%%%%%%%%%%%%%%%%%%%%%%%%%%%%%%%%%%%%%%%%%%%%%%%%%%%%%%
$O^{+}(NS(X))$ is a geometrically finite subgroup of the first kind.
By Theorem \ref{hyperbolic_geometry6},
we have the following:
\begin{lem}\label{generalized polytope}
Let $X$ be a K3 surface.
Then a fundamental domain of $\mathbb{H}^{n}_{X}$ for $O^{+}(NS(X))$ is a generalized polytope.
In particular, the action of $O^{+}(NS(X))$ on $\mathbb{H}^{n}_{X}$ is finite covolume.
\end{lem}

%%%%%%%%%%%%%%%%%%%%%%%%%%%%%%%%%%%%%%%%%%%%%%%%%%%%%%%%%%%%%%%%%%%%%%%%%%%
\begin{lem}\label{generalized polytope2}
Let $a \in A_{X}$,
and let $P$ be the Dirichlet domain, centered at $a$, of $\mathbb{H}^{n}_{X}$ for $O^{+}(NS(X))$.
Then, $\operatorname{Aut}(A(X))\cdot \overline{P} = \overline{A_{X}}$,
where $\overline{P}$ (resp. $\overline{A_{X}}^{H}$)
are the closure of $P$ (resp. $A_{X}$) in $\mathbb{H}^{n}_{X}.$
\end{lem}

\begin{proof}
By Lemma \ref{ample cone vs Dirichlet domain} and the definition of $P$,
we have 
\begin{align*}
P &=\{ x\in \mathbb{H}^{n}_{X}: d(x,a) <d(x, ga) {\text{ for any }} g\in O^{+}(NS(X)) \}, \\
A_{X} &= \{ x\in \mathbb{H}^{n}_{X}: d(x,a) <d(x, wa) {\text{ for any }} w \in W(X) \}.
\end{align*}
Then we have $P \subset A_{X}$ and so $\overline{P} \subset \overline{A_{X}}^{H}.$
Hence we obtain $\operatorname{Aut}(A(X))\cdot \overline{P} \subset \overline{A_{X}}^{H}.$

Conversely,
because the tessellation $\{g\overline{P}: g\in O^{+}(NS(X))\}$ of $\mathbb{H}^{n}_{X}$ is locally finite, the union 
$\operatorname{Aut}(A(X))\cdot\overline{P}$
is a closed subset.
Then it is enough to show that $A_{X}\subset \operatorname{Aut}(A(X))\cdot\overline{P}.$
For any $x\in A_{X}$, there exist $w \in W(X)$ and $\gamma \in \operatorname{Aut}(A(X))$
such that $x \in w\gamma \overline{P}.$
Then we have $x \in w\gamma \overline{P}\cap A_{X} \subset w\overline{A_{X}}^{H}\cap A_{X}.$
By \cite[Theorem 6.6.4]{Ratcliffe},
we have $w=1$. In particular, we have $x\in \gamma\overline{P}$.
Hence we have $A_{X}\subset \operatorname{Aut}(A(X))\cdot\overline{P}$.
\end{proof}
%%%%%%%%%%%%%%%%%%%%%%%%%%%%%%%%%%%%%%%%%%%%%%%%%%%%%%%%%%%%%%%%%%%%%%%%%%%

\begin{defi}\label{the sets of bounded parabolic points}
Let $P$ be the same as in Lemma \ref{generalized polytope2}.
We define sets of bounded parabolic points of $O^{+}(NS(X))$ as follows:

\begin{itemize}
	\item $I_{P} =$ the set of ideal vertices of $P$,
	
	\item $V = Aut(A(X))\cdot I_{P}$,
	
	\item $V_{NS}= O^{+}(NS(X))\cdot I_{P},$
	
	\item $C_{X} = \{c \in V : \operatorname{vcd} W(X)_{c} < n-1 \}.$
\end{itemize}
\end{defi}

\begin{defi}\label{limit set of X}
Let $X$ be a K3 surface.
Since $\operatorname{Aut}_{s}(X)$ is of finite index in $\operatorname{Aut}(A(X))$,
the limit set of $\operatorname{Aut}_{s}(X)$ is equal to that of $\operatorname{Aut}(A(X))$.
We denote this limit set by $\Lambda_{X}$,
and is called the {\it{limit set of the K3 surface $X$.}}
We denote the hyperbolic convex hull 
of $\Lambda_{X}$ by $C(X)$,
and put $B(X) = C(X) \cap \mathbb{H}^{n}_{X}$.
\end{defi}

\begin{lem}\label{horoballs}
For any $c\in V\backslash C_{X}$,
there exists a horoball $U(c)$ based at $c$ such that $B(X)\cap U(c) =\varnothing$.
Moreover, the point $c$ is isolated in the boundary $\partial A_{X}$ associated with the ample cone.
\end{lem}
\begin{proof}
If $\operatorname{Aut}(A(X))$ is elementary, this lemma is obvious.

Suppose that  $\operatorname{Aut}(A(X))$ is non-elementary.
Let $\phi :\mathbb{H}^{n}_{X} \rightarrow U^{n}$ be an isometry 
from $\mathbb{H}^{n}_{X}$ to the upper half-space model $U^{n}$
such that $\phi(c) = \infty.$
The isometry $\phi$ induces the natural homeomorphisms from 
$\partial \mathbb{H}^{n}_{X}$
to $\partial U^{n}.$ We also denote this homeomorphism by $\phi$.
Let $b$ be a point of $U^{n}$ corresponding to a point $a$ in $A_{X}$.
We can consider the stabilizer group $W(X)_{\infty}$ as the Poincar\'{e} extension of a subgroup of
$I(E^{n-1})$.
Since $\operatorname{vcd} W(X)_{c} = n-1$, $W(X)_{\infty}$ is a crystallographic group of $E^{n-1}$.
Then the Dirichlet domain $K$ of $E^{n-1}$ for $W(X)_{\infty}$, centered at $b$, is compact.
Proposition \ref{ample cone vs Dirichlet domain} implies that
$\Lambda_{X} \subset \partial A_{X}.$
Then we have $\phi(\Lambda_{X}) \subset K\cup \{\infty \}.$ 
Since $\operatorname{Aut}(A(X))$ is non-elementary, the limit set $\Lambda_{X}$ is perfect:
i.e., every point of $\Lambda_{X}$ is not isolated.
Hence the infinity $\infty$ is not in $\phi(\Lambda_{X})$.
By the compactaness of $K$, there exists a horoball $B_{\infty}$ based at $\infty$
such that $C(K)\cap B_{\infty} = \varnothing,$
where $C(K)$ is the hyperbolic convex hull of $K$.
In particular, $B(X) \cap B_{\infty}=\varnothing.$
The inverse image $\phi^{-1}(B_{\infty})$ of $B_{\infty}$ by $\phi$ is the desired horoball.
\end{proof}

%%%%%%%%%%%%%%%%%%%%%%%%%%%%%%%%%%%%%%%%%%%%%
%%%%%%%%%%%%%%%%%% Section {Mordell-Weil groups and maximal parabolic subgroup %%%%%%

%%%%%%%%%%%%%%%%%%%%%%%%%%%%%%%%%%%%%%%%%%%%%%%%%%%%%%%%%%%
\begin{lem}\label{the rank of MW}
Let $e$ be an elliptic divisor of $X$,
let $f_{e}$ be an elliptic fibration associated to $e$,
and let $\operatorname{Aut}(X)_{e}$ be the stabilizer group of $e$.
Then 
$$
\operatorname{vcd{\operatorname{Aut}(X)_{e}}}
=  \operatorname{vcd} \operatorname{MW}(f_{e}) {\text{, and }}
\operatorname{vcd}W(X)_{e}
= \sum_{v\in \mathbb{P}^{1}}{(m_{v}-1)},$$
where $\mathrm{MW}(f_{e})$ is the Mordell-Weil group of $f_{e}$,
and 
$m_{v}$ be the number of connected components of 
the fiber $f^{-1}(v)$ for $v\in \mathbb{P}^{1}.$
\end{lem}

%%%%%%%%%%%%%%%%%%%%%%%%%%%%%%%%%%%%%%%%%%%%%%%%%%%%%%%%%%%
\begin{proof}
Let $v$ be a point in $\mathbb{P}^{1},$
and let $m_{v}$ be the number of connected components of the fiber $f^{-1}(v)$.
By Shioda-Tate formula, 
$$ \rho(X)-2 = 
\operatorname{vcd}{MW(f_{e})} + \sum_{v\in \mathbb{P}^{1}}{(m_{v}-1)}.$$

If $m_{v}\neq 1$,
we denote by $W_{v}$ the Weyl group generated by the connected components of $f^{-1}(v)$,
which is of type $\tilde{A}$, $\tilde{D}$, or $\tilde{E}.$
Since $W_{v}$ fixes the elliptic divisor $e$, we have $\oplus W_{v} \subset W(X)_{e}.$
Here, for each $v$ with $m_{v}\neq1$, the Weyl group $W_{v}$ acts on the $(m_{v}-1)$-dimensional
Euclidean space, and admits
an $(m_{v}-1)$-dimensional simplex $\Delta_{v}$ as a fundamental domain for the action of $W_{v}.$
Hence $\oplus W_{v}$ acts on the $(\sum_{v\in \mathbb{P}^{1}}{(m_{v}-1})\big)$-dimensional Euclidean space,
and has the fundamental domain $\prod_{v}{\Delta_{v}}$.
Then we have
$$ \sum_{v \in \mathbb{P}^{1}}(m_{v}-1)
= \operatorname{vcd}(\oplus W_{v})
\leq \operatorname{vcd}W(X)_{e}.$$
Since the Mordell-Weil group of $f_{e}$ is a subgroup of $\operatorname{Aut}(X)_{e}$,
then 
$\operatorname{rk} \operatorname{MW}f_{e} \leq \operatorname{vcd}\operatorname{Aut}(X)_{e}$, 
where $\operatorname{rk} \operatorname{MW} f_{e}$ is the rank of Mordell-Weil group of $f_{e}$.

Since both $W(X)_{e}$ and $\operatorname{Aut}(X)_{e}$ fix the elliptic divisor $e$, which is a
cusped limit points of $O^{+}(NS(X))$, these two groups are elementary of 
either elliptic or parabolic type. Therefore,
there exists an integer $l$ (resp. $m$) such that $W(X)_{e}$ (resp. $\operatorname{Aut}(X)_{e}$)
has a subgroup that is isomoprhic to $\mathbb{Z}^{l}$ (resp. $\mathbb{Z}^{m}$) 
of finite index in $W(X)_{e}$ (resp. $\operatorname{Aut}(X)_{e}$).
Hence we have 
\begin{align}
\sum_{v\in \mathbb{P}^{1}}{(m_{v}-1)} \leq l, \\
\operatorname{rk} \operatorname{MW} f_{e} \leq m,
\end{align}
so that $\rho (X)-2 \leq l+m.$
On the other hand, $O^{+}(NS(X))_{e}$ has a subgroup that is isomorphic to $\mathbb{Z}^{l+m}.$
Then $l+m = \rho(X)-2.$ The inequality 6.2.1 and 6.2.2 imply that

\begin{align*}
l = \sum_{v \in \mathbb{P}^{1}}(m_{v}-1), \\
m = \operatorname{rk} \operatorname{MW} f_{e}.
\end{align*}
These are namely the desired equalities.
\end{proof}
%%%%%%%%%%%%%%%%%%%%%%%%%%%%%%%%%%%%%%%%%%%%%%

By Lemma \ref{the rank of MW},
we have
$C_{X} 
= \{c \in V : \operatorname{vcd}\operatorname{Aut}(X)_{c}  >1 \}
= \{c \in V : \operatorname{vcd}\operatorname{Aut}(A(X))_{c}  >1 \}.$ 
Then each element of $C_{X}$ is fixed by some parabolic transformation of 
$\operatorname{Aut}(A(X))_{c}.$

%%%%%%%%%%%%%%%%%%%%%%%%%%%%%%% Prop
As noted in the introduction, 
the following theorem is also found by Kohei Kikuta at about the same time
(see the introduction for details).

\begin{thm}\label{geometrically finiteness}
Let $X$ be a K3 surface,
and let $\operatorname{Aut}_{s}(X)$ be the symplectic automorphism group of $X$.
Then $\operatorname{Aut}_{s}(X)$ is geometrically finite.
\end{thm}

\begin{proof}
By Theorem \ref{aut1},
the symplectic automorphism group $\operatorname{Aut}_{s}(X)$
is a subgroup of finite index in $\operatorname{Aut}(A(X)).$
Then it is enough to show that $\operatorname{Aut}(A(X))$ is geometrically finite.
Recall the notation of $P$, $V$, $V_{NS}$, and $C_{X}$
in Definition \ref{the sets of bounded parabolic points}

By Theorem \ref{Ratcliffe_thm_12.6.5},
there exist proper horoballs at $c \in V_{NS}$ such that

\begin{enumerate}
	\item $\{ \overline{U(c)} : c \in V_{NS}\}$ are mutually disjoint,
	\item $gU(c) = U(gc)$ for any element $g \in O^{+}(NS(X))$,
	\item $\{\overline{U(c)}\backslash {c} \} $ is locally finite in $\mathbb{H}^{n}_{X}$.
\end{enumerate}
By Lemma \ref{horoballs}, shrinking the horoballs $U(c)$ sufficiently small,
we obtain that 
$$ B(X)\backslash \bigcup_{c\in C_{X}} U(c) 
= B(X)\backslash \bigcup_{c\in V} U(c).$$

By applying Lemma \ref{generalized polytope2},
\begin{align*}
B(X)\backslash \bigcup_{c\in C_{X}} U(c) 
&\subset  \hspace{1mm}\overline{A_{X}}^{H}\backslash \bigcup_{c\in V} U(c) \\
&= \bigl( \bigcup_{g\in \operatorname{Aut}(A(X))}\hspace{-1mm}g\overline{P}\hspace{1mm}\bigr)
\big\backslash \bigcup_{c\in V} U(c) \\
&= \bigcup_{g\in \operatorname{Aut} (A(X))} \big(g\overline{P}
\hspace{1mm} 
\backslash \bigcup_{c\in V} U(c) \big)\\
&=  \bigcup_{g \in  \operatorname{Aut}(A(X))}g(\overline{P}\backslash \bigcup_{c \in V} U(c)).
\end{align*}
Hence 
we have
\begin{equation}\label{convexcore}
B(X)\backslash \bigcup_{c\in C_{X}} U(c) 
\subset 
\bigcup_{g \in  \operatorname{Aut}(A(X))}g(\overline{P}\backslash \bigcup_{c \in V} U(c)).
\end{equation}

Let $\pi:\mathbb{H}^{n}_{X} \rightarrow \mathbb{H}^{n}_{X}/\hspace{-1mm}\operatorname{Aut}(A(X))$
be the natural quotient map.
We define 
$V_{c} = U(c)/\hspace{-1mm}\operatorname{Aut}(A(X)).$
By the equality \ref{convexcore}, we obtain
$$ (B(X)/\operatorname{Aut}(A(X)))\backslash \bigcup_{c\in I_{P}\cap C_{X}}V_{c} 
\subset \pi(\overline{P}\backslash \bigcup_{c\in V} U(c)).$$
Then 
$(B(X)/\operatorname{Aut}(A(X)))\backslash \bigcup_{c\in I_{P}\cap C_{X}}V_{c}$
 is compact and $\{ V_{c}\}_{c\in I_{P}\cap C_{X}}$
is a set of finitely many horocusps with disjoint closures.
By Theorem \ref{hyperbolic_geometry7}, $\operatorname{Aut}(A(X))$ is geometrically finite.
\end{proof}

\begin{prop}\label{horopoints of Ax}
Let $p$ be a cusped limit point of $\operatorname{Aut}(A(X)).$
Then $p$ is a horopoint of $\overline{A_{X}},$
i.e.,
there exists a horoball $U$ based at $p$ such that $U$ 
meets just the sides of $\overline{A_{X}}$ incident with $p$.
\end{prop}
\begin{proof}
We now pass to the upper half-space model $U^n$ 
and conjugate $\operatorname{Aut}(A(X))$ so that $p=\infty.$
Let $P$ be a Dirichlet polyhedron $P$, centerd at $a\in A_{X}$,
for $O^{+}(NS(X)).$
From the proof of \cite[Theorem 12.4.4]{Ratcliffe},
there exist a horoball $\Sigma$ based at $\infty$
and elements $g_{1}, \cdots, g_{k}$ in $O^{+}(NS(X))$
such that 
$$\Sigma \subset 
\bigcup_{f\in \operatorname{Aut}(A(X))_{\infty}}(\bigcup_{i=1}^{k}g_{i}P)$$
and $\Sigma$ meets just the vertical sides of $fg_{i}P$.

Let $S$ be a side of $A_{X}$ with $S \cap \Sigma \neq \varnothing.$
There exists an element $x\in S$ and $r>0$ such that
$B(x, r)\cap S $ is homeomorphic to $B^{n-1}.$
Since $P$ is locally finite, 
shrinking $r$,
there exists $g_{1}, \cdots, g_{k} \in O^{+}(NS(X))$ such that $x\in g_{i}\overline{P}$
and 
$B(x, r) \subset \bigcup_{i=1}^{k}g_{i}\overline{P}.$
In particular,
we have 
$$B^{n-1} \cong B(x, r)\cap S
\subset
\bigcup \big\{ g_{i}T: g_{i}\in \operatorname{Aut}(A(X)),
T {\text{ is a side of }} \overline{P},
{\text{and }} x\in gT
\big\}.$$
Hence there exists a side $T$ of $g\overline{P}$ such that 
$T^{\circ} \cap S\neq \varnothing.$
By \cite[Theorem 6.3.8]{Ratcliffe}, we have $T\subset S.$
By $x\in T\cap \Sigma$, then $T$ is a vertical side of $g\overline{P}.$
Therefore, we have $\infty \in T \subset S$.
\end{proof}

%%%%%%%%%%%%%%%%%% Section {Mordell-Weil groups and maximal parabolic subgroup %%%%%%

There is a natural one-to-one correspondence from the set of maximal parabolic subgroups
of $\operatorname{Aut}_{s}(X)$ to the rational ray of $\partial A(X)$.
By multiplying a positive integer, we have a primitive element in $NS(X)$.
From this argument and Lemma \ref{the rank of MW},
we have the following proposition:
\begin{prop}\label{P_vs_MW}
Let $P$ be a maximal parabolic subgroup of $\operatorname{Aut}_{s}(X).$
Then there exists an element $w \in W(X)$ such that the fixed point of $wPw^{-1}$
is an elliptic divisor $e$. Moreover, the rank of $P$ is equal to the rank of
the Mordell-Weil group of the elliptic fibration $f_{e}$ associated with $e$.
\end{prop}

%By Theorem \ref{geometrically finiteness} and Proposition \ref{P_vs_MW},
%we have the following corollary.
%\begin{cor}
%Let $\operatorname{MW}(X)_{\geq 1}$ be the complete representative of the conjugacy classes %of the Mordell-Weil groups of $X$. If $\operatorname{Aut}_{s}(X)$ is non-elementary,
%then $\operatorname{Aut}(X)$ is hyperbolic relative to $\operatorname{MW}(X)_{\geq 1}$.
%\end{cor}

%%%%%%%%%%%%%%%%%%%%%%%%%%%%%%%%%%%%%%%%%%%%%%%%%%%%%%%%%%%%%%
\subsection{The blown-up boundary of ample cones of K3 surfaces}

The notion of the blown-up coronae for relatively hyperbolic groups 
(a generalization of geometrically finite groups)
is defined in the paper of Fukaya-Oguni \cite{Fukaya-Oguni},
and Fukaya \cite{Fukaya} established Bestvina-Mess type formula for relatively hyperbolic groups.
We can immediately obtain Theorem 1 by applying Fukaya's result \cite[Corollary B]{Fukaya}.

For readability to algebraic geometers,
in this section, 
we give a simplified version of Fukaya's result,
using the geometry of ample cones of K3 surfaces and hyperbolic geometry. 
Note that the proof method and some definitions in this section are not itself new.

%%%%%%%%%%%%%%%%%%%%%%%%%%%%%%%%%%%%%%%%%%%%%%%%%%%%%%%%%%%%%%
\begin{defi}\label{blowing-up at a point}
Let $U$ be an open subset of $\mathbb{R}^{n}$,
and let $p$ be a point in $U$.
We denote the $(n-1)$-dimensional unit sphere in $\mathbb{R}^{n}$ by $S^{n-1}$.
We put 
$$\operatorname{Bl}_{p}(U)
= \{ (x, t) \in U\times S^{n-1}: |x-p|t = (x-p) \},$$
where $| \cdot |$ is the Eulidean norm.
We define a function $\pi_{p}: \operatorname{Bl}_{p}(U) \rightarrow U$ by setting 
$\pi_{p}(x, t) = x.$
We call the function $\pi_{p}$ the {\it{blowing-up of $U$ at a point $p$ (to $S^{n-1}$).}}
For a subset $A$ of $U$, the closure of  $\pi_{p}^{-1}(A\backslash \{p\})$
in $\mathbb{R}^{n}$
is called
the {\it{strict transform of $A$ at a point $p$}},
and is denoted by $\operatorname{St}_{p}(A)$.

\end{defi}
Since $\pi_{p}$ restricts to a homeomorphism
from $\mathrm{Bl}_{p}(U)\backslash \pi^{-1}_{p}(p)$ to $U\backslash \{p\},$
we identify these two spaces.
%%%%%%%%%%%%%%%%%%%%%%%%%%%%%%%%%%%%%%%%%%%%%%%%%%%%%%%%%%%%%%
\begin{defi}\label{blowing-up at points}
Let $\mathcal{P}=\{p_{1}, \dots, p_{m}\}$ be a subset of $U$.
For each $p\in \mathcal{P}$, we take an open neighborhood $U_{p}$ of $p$ 
such that $U_{p} \cap \mathcal{P} = \{p\}.$
Let $\pi_{p_{i}}: \operatorname{Bl}_{p_{i}}(U_{p_{i}}) \rightarrow U_{p_{i}}$
be the blowing-up of $U_{p_{i}}$ at a point $p_{i}$.
We define $\mathrm{Bl}_{\mathcal{P}}(U)$ by 
$$\mathrm{Bl}_{\mathcal{P}}(U)
= (U\backslash \mathcal{P}) 
\amalg 
({\coprod_{p\in \mathcal{P}} \mathrm{Bl}_{p}(U_{p})})/\sim,$$
where the equivalence relation $\sim$ is generated by $x\sim \pi_{p_{i}}(x)$
for all
$x \in \coprod_{p \in \mathcal{P}}
\pi_{p}^{-1}(U_{p}\backslash \lbrace p\rbrace).$

We denote by $\lbrack x \rbrack$ the class in
$\mathrm{Bl}_{\mathcal{P}}(U)$ corresponding to $x$.
We define the function $\pi_{\mathcal{P}}: \operatorname{Bl}_{\mathcal{P}}(U) \rightarrow U$
by setting 
$$ \pi_{\mathcal{P}}(\lbrack x \rbrack) = 
\left\{
\begin{array}{ll}
x  &  (x \in U\backslash \mathcal{P}) \\
\pi_{p}(x) & ( x \in \operatorname{Bl}_{p}(U_{p}), p\in \mathcal{P})
\end{array}
\right.$$
We call the function $\pi_{\mathcal{P}}$ the 
{\it{blowing-up of $U$ at a subset $\mathcal{P}$}}.
For a subset $A$ of $U$, the closure of  
$\pi_{\mathcal{P}}^{-1}(A\backslash \mathcal{P})$ in $\mathbb{R}^{n}$
is called
the {\it{strict transform of $A$ at a subset $\mathcal{P}$}},
and is denoted by $\operatorname{St}_{\mathcal{P}}(A)$.
\end{defi}

%%%%%%%%%%%%%%%%%%%%%%%%%%%%%%%%%%%%%%%%%%%%%%%%%%%%%%%%%%%%%%
\begin{defi}\label{blowing-up at bounded parabolic points}
Let $\Gamma$ be a geometrically finite subgroup of $M(B^{n})$,
and let $C$ be the set of cusped limit points of $\Gamma$.
For finite subsets $\mathcal{P}$ and $\mathcal{Q}$ of $C$ 
with $\mathcal{P}\subset \mathcal{Q}$,
there exists the narural projection 
$\pi_{\mathcal{P}, \mathcal{Q}}:
\operatorname{Bl}_{\mathcal{Q}}(\mathbb{R}^{n})\rightarrow
\operatorname{Bl}_{\mathcal{P}}(\mathbb{R}^{n})$
such that $\pi_{\mathcal{P}}\circ \pi_{\mathcal{P}, \mathcal{Q}} = \pi_{\mathcal{Q}}.$

We denote the projective limit of
($\operatorname{Bl}_{\mathcal{P}}(\mathbb{R}^{n})$, $\pi_{\mathcal{P}, \mathcal{Q}}$)
by
$\pi_{C} = \varprojlim \pi_{\mathcal{P}}
: \operatorname{Bl}_{C}(\mathbb{R}^{n})\rightarrow  \mathbb{R}^{n}.$
We call the function $\pi_{C}$ the {\it{blowing-up of $\mathbb{R}^{n}$ at $C$}}.
For a subset $A$ of $\mathbb{R}^{n}$,
the closure of  
$\pi_{C}^{-1}(A\backslash \mathcal{P})$ in $\mathbb{R}^{n}$
is called
the {\it{strict transform of $A$ at $C$}},
and is denoted by $\operatorname{St}_{C}(A)$.
The strict transform of $\overline{B^{n}}$ (resp. $\partial B^{n}$) is denoted by
$\overline{B}^{bl}$ (resp. $\partial B^{bl}$).
\end{defi}

%%%%%%%%%%%%%%%%%%%%%%%%%%%%%%%%%%%%%%%%%%%%%%%%%%%%%%%%%%%%%%
\begin{lem}
The action of a gemetrically finite subgroup $\Gamma$ of $M(B^{n})$ on $\mathbb{R}^{n}$
(resp. $B^{n}$, $\partial B^{n}$) is extended to an action on $\operatorname{Bl}_{C}(\mathbb{R}^{n})$
(resp. $\overline{B}^{bl}$, $\partial B^{bl}$).
\end{lem}
\begin{proof}
Since $\Gamma$ preserves the cusped limit point of $\Gamma$, 
the action of $\Gamma$ on $\mathbb{R}^{n}$ can be extend to that on $\operatorname{Bl}_{C}(\mathbb{R}^{n}).$ As $\Gamma$ preserves $\overline{B^{n}}$ and $\partial B^n$, $\Gamma$ acts on $\overline{B}^{bl}$ 
(resp. $\partial B^{bl}$)
\end{proof}

%%%%%%%%%%%%%%%%%%%%%%%%%%%%%%%%%%%%%%%%%%%%%%%%%%%%%%%%%%%%%%
\begin{defi}
Let $\Gamma$ be a geometrically finite subgroup of $M(B^{n})$.
A point $a$ of $\partial B^{bl}$ is a {\it{limit point of $\Gamma$ in $\partial B^{bl}$ }}
if there is a point $x\in B^{n}$ and there exists a sequence $\{g_{i}\}_{i\in \mathbb{N}}$ of $\Gamma$
such that $a= \lim g_{i}x$.
The set of limit points of $\Gamma$ is called {\it{the blown-up limit set of $\Gamma$}}, 
and is denoted by $\Lambda^{bl}(\Gamma).$
\end{defi}

%%%%%%%%%%%%%%%%%%%%%%%%%%%%%%%%%%%%%%%%%%%%%%%%%%%%%%%%%%%%%%
\begin{lem}\label{strict transform of lines}
Let $\eta: U^{n} \rightarrow B^{n}$ be the standard transformation
from $U^{n}$ to $B^{n}$.
For non-zero vector $v\in \partial U = E^{n-1}$ and $a\in \overline{U^{n}}$,
we define 
a half line $l:\lbrack 0, \infty ) \rightarrow \overline{U^{n}} $ by
$l(t) = tv+a$, and put $L(t)= (\eta \circ l)(t).$
Then 
$$\operatorname{St}_{e_{n}}(L(t)) \cap \pi^{-1}(e_{n})= \Bigl\{(e_{n}, \frac{v}{|v|})\Bigr\}
\subset \operatorname{Bl}_{e_{n}}(\mathbb{R}^{n}).$$
In particular, the intersection of $\operatorname{St}_{e_{n}}(L(t))$ with $\pi^{-1}(e_{n})$
is determined by the direction of $v$.
\end{lem}
\begin{proof}
By direct calculation, we obtain
$$ L(t)- e_{n} = \frac{2(tv+\rho(a)-e_{n})}{|tv+\rho(a)-e_{n}|^{2}}.$$
Then
$$ \lim_{t\rightarrow \infty} \frac{L(t)- e_{n}}{|L(t)- e_{n}|}
=\lim_{t\rightarrow \infty}\frac{tv+\rho(a)-e_{n}}{|tv+\rho(a)-e_{n}|}
= \frac{v}{|v|}.$$
Hence, $\displaystyle \operatorname{St}_{e_{n}}(L) = L\cup \{ (e_{n}, \frac{v}{|v|})\}.$
\end{proof}

%%%%%%%%%%%%%%%%%%%%%%%%%%%%%%%%%%%%%%%%%%%%%%%%%%%%%%%%%%%%%
\begin{cor}\label{strict transform of subspaces}
Let $V$ be a vector subspace of $\mathbb{R}^{n}$, 
and let $a$ be a point of $\overline{U^{n}}$.
Let $A$ be an open $r$-neighborhood of $V+a$ in $U^n$ for a positive number $r$.
Then,
$$\operatorname{St}_{e_{n}}(\eta(A)) \cap \pi^{-1}_{e_{n}}(e_{n})
= 
\Bigl\{
(e_{n}, \frac{v}{|v|}) \in \operatorname{Bl}_{e_{n}}(\mathbb{R}^{n})
: v\in V\backslash \{0\}
 \Bigr\} 
 \cong S^{r-1}, $$
 where $r$ is the dimension of the vector space $V$.

\end{cor}
%%%%%%%%%%%%%%%%%%%%%%%%%%%%%%%%%%%%%%%%%%%%%%%%%%%%%%%%%%%%%%
\begin{lem}\label{limit set at a parabolic point}
Let $\Gamma$ be a geometrically finite subgroup of $M(B^{n})$.
For every cusped limit point $p$ of $\Gamma$,
we have $$\Lambda^{bl}(\Gamma) \cap \pi_{C}^{-1}(p) =\Lambda^{bl}(\Gamma_{p}) \cong S^{r-1},$$
where $r$ is the rank of the stabilizer group $\Gamma_{p}$, i.e., the rank of maximal abelian subgroup of $\Gamma_{p}$.
\end{lem}
\begin{proof}
By conjugating $\Gamma$ in $M(B^{n})$,
we can suppose $p=e_{n}.$
Let $\eta: U^{n} \rightarrow B^{n}$ be the standard transformation.
The group $\Gamma$ acts on $U^n$ by setting
$\gamma (x) = (\eta^{-1}\gamma \eta) (x) $ for any $\gamma \in \Gamma.$

First, we show that $\Lambda^{bl}(\Gamma_{p}) \cong S^{r-1}.$
There is a $\Gamma_{\infty}$-inveriant $r$-plane $Q$ of $E^{n-1}$ such that
$Q/\Gamma_{\infty}$ is compact.
For any $r>0$, the $r$-neighborhood $N(Q,r)$ of $Q$ in $U^{n}$ is $\Gamma_{\infty}$-invariant.
The $r$-plane $Q$ has a description of the form $Q=V+a,$
where $V$ is the $r$-dimensional vector subspace of $E^{n-1}$ 
and $a$ is a vector in $Q$.
By Corollary \ref{strict transform of subspaces},
we have
$$\Lambda^{bl}(\Gamma_{p}) \subset 
\{ (e_{n}, \frac{v}{|v|}) \in
\operatorname{Bl}_{e_{n}}(\mathbb{R}^{n}: 
v \in V\backslash \{0\} \}.$$

Conversely,
for $v\in V\backslash \{0\}$ and any integer $n>0$,
there is an element $\gamma_{n}$ of $\Gamma_{\infty}$ such that
$\gamma_{n} a \in N(nv+a, R),$
where $R$ is the Euclidean diameter of a fundamental domain of $Q$ for $\Gamma_{\infty}.$
By Lemma \ref{strict transform of lines},
$$\lim_{n\rightarrow \infty} \gamma_{n}x = \frac{v}{|v|} \in \Lambda^{bl}(\Gamma_{p}).$$
Hence, we have
$$ \Lambda^{bl}(\Gamma_{p}) =
\{ (e_{n}, \frac{v}{|v|}) \in \operatorname{Bl}_{e_{n}}(\mathbb{R}^{n}): v\in V\backslash \{0\}\} \cong S^{r-1}$$

Secondly, we show that $\Lambda^{bl}(\Gamma) \cap \pi_{C}^{-1}(p) =\Lambda^{bl}(\Gamma_{p}).$
Let $a$ be a point in $\Lambda^{bl}(\Gamma) \cap \pi_{C}^{-1}(p).$
There exist a point $x\in B^{n}$ and a sequence $\{g_{i}\}$ of $\Gamma$ such that
$\displaystyle a=\lim_{n\rightarrow \infty} g_{i}x.$
If the sequence $\{g_{i}\}$ has infinitely many elements of $\Gamma_{p}$,
by the definition of $\Lambda^{bl}(\Gamma_{p})$, we have $a\in \Lambda^{bl}(\Gamma_{p}).$
Otherwise, there exists a proper cusped region $U(Q, r)$ such that
$g_{i}x \in N(Q,r)$ for infinitely many $i\in \mathbb{N}.$
By Corollary \ref{strict transform of subspaces}, $a$ is in $\Lambda^{bl}(\Gamma_{p}).$
Hence we have $\Lambda^{bl}(\Gamma) \cap \pi_{C}^{-1}(p) \subset \Lambda^{bl}(\Gamma_{p}).$
The converse inclusion is obvious.
\end{proof}

%%%%%%%%%%%%%%%%%%%%%%%%%%%%%%%%%%%%%%%%%%%%%%%%%%%%%%%%%%%%%%
\begin{lem}\label{limit set}
For any $x\in B^{n}$, 
$$ \Lambda^{bl}(\Gamma) = \overline{\Gamma x} \cap \partial B^{bl}.$$
In particular,
the blown-up limit set of $\Gamma$ 
is a $\Gamma$-inveriant closed subset of $B^{bl}$.
\end{lem}
\begin{proof}
By the definition of $\Lambda^{bl}(\Gamma)$, 
it is obvious that 
$\Lambda^{bl}(\Gamma) \supset \overline{\Gamma x} \cap \partial B^{bl}.$

Conversely,
for any $a\in \Lambda^{bl}(\Gamma)$, we put $\overline{a}=\pi_{C}(a) \in \Lambda (\Gamma)$.
If $a$ does not lie in the set $C$ of cusped limit points of $\Gamma$, we can show that $a$ is a limit point of $\Gamma$ in $\partial B^{bl}$ 
in the same way as a limit point of $\Gamma$ in $\partial B^{n-1}$.
If $a\in C$, applying Lemma \ref{limit set at a parabolic point} to the cusped limit point $a$,
$$
a \in \Lambda^{bl}(\Gamma)\cap \pi^{-1}_{C}(\overline{a})
= \Lambda^{bl}(\Gamma_{\overline{a}}) 
= \overline{\Gamma_{\overline{a}}x}\cap \partial B^{bl} 
\subset \overline{\Gamma x}\cap \partial B^{bl}.
$$
Therefore $\Lambda^{bl} (\Gamma) \subset \overline{\Gamma x}\cap \partial B^{bl}.$
\end{proof}

%%%%%%%%%%%%%%%%%%%%%%%%%%%%%%%%%%%%%%%%%%%%%%%%%%%%%%%%%%%%%%
\begin{lem}\label{the blowing-up boundary equals strict transform}
Let $\Gamma$ be a non-elementary geometrically finite subgroup of $M(B^{n})$.
Then, $$\Lambda^{bl}(\Gamma) = \operatorname{St}_{C}(\Lambda(\Gamma)).$$

\end{lem}
\begin{proof}
By Lemma \ref{limit set}, we have $\Lambda^{bl}(\Gamma) 
\supset 
\overline{\pi^{-1}_{C}(\Lambda(\Gamma)\backslash C)}
=
\operatorname{St}_{C}(\Lambda(\Gamma)).$

Conversely,
it is enough to show that
$ \operatorname{St}_{C}(\Lambda(\Gamma)) \cap \pi^{-1}_{C}(p)
\subset \Lambda^{bl}(\Gamma_{p}) $ for each $p \in C.$
By conjugating $\Gamma$ in $M(B^{n})$, we can suppose $p=e_{n}$.
As the similar way in Lemma \ref{limit set at a parabolic point},
we pass the upper half-space model $U^n$, and $\Gamma$ acts on $U^n$.
Let $U(Q,r)$ be a proper cusped region based at $\infty$.
Then $\Lambda(\Gamma) \subset N(Q, r) \cup \{\infty\}.$
There exist a vector subspace $V$ of $E^{n-1}$ and a vector $a$ of $Q$
such that $Q=V+a.$
By Corollary \ref{strict transform of subspaces} 
and Lemma \ref{limit set at a parabolic point}
we have 
$$ \operatorname{St}_{C}(\Lambda(\Gamma)) \cap \pi^{-1}_{C}(p)
\subset \big\{(e_{n}, \frac{v}{|v|}) : v\in V\backslash \{0\}\big\} 
= \Lambda^{bl}(\Gamma_{p})
$$
\end{proof}
%%%%%%%%%%%%%%%%%%%%%%%%%%%%%%%%%%%%%%%%%%%%%%%%%%%%%%%%%%%%%%
%\begin{lem}
%Let $\Gamma$ be a discrete subgroup of $M(B^n)$,
%let $C$ be the set of cusped limit points of $\Gamma$.
%Let a collection of horoballs $U(c)$ baset at $c\in C$ be the same as in \cite{Ratcliffe} Theorem %12.6.5.
%Then the Euclidean diameter of $U(c)$ goes to zero: i.e.,
%for any $\epsilon >0$, 
%there exists a finite subset $C_{\epsilon}$ of $C$ such that
%$\operatorname{diam}U(c) < \epsilon$ for any $c\in C\backslash C_{\epsilon}$.
%
%
%
%\end{lem}
%%%%%%%%%%%%%%%%%%%%%%%%%%%%%%%%%%%%%%%%%%%%%%%%%%%%%%%%%%%%%%

\begin{defi}
Let $X$ be a K3 surface,
and let $V$ (resp. $C_{X}$) be 
the same as in Definition \ref{the sets of bounded parabolic points}.
We denote the strict transform of $\partial A_{X}$ (resp. $\overline{A_{X}}$) at $C_{X}$
by $\partial A_{X}^{bl}$ (resp. $\overline{A_{X}}^{bl}$).
\end{defi}
%%%%%%%%%%%%%%%%%%%%%%%%%%%%%%%%%%%%%%%%%%%%%%%%%%%%%%%%%%%%%%

The following result is given by Baragar \cite{Baragar2}.
By using the geometrically finiteness of $O^{+}(NS(X))$,
we can give another proof of the result.
%%%%%%%%%%%%%%%%%%%%%%%%%%%%%%%%%%%%%%%%%%%%%%%%%%%%%%%%%%%%%%
\begin{thm}[ {\cite[Theorem 3.3]{Baragar2} }]\label{Baragar's result for limit sets}
For a K3 surface $X$, 
the set $\partial A_{X} \backslash \Lambda_{X}$ 
is equal to 
$V\backslash C_{X}
=\{c\in V: \operatorname{Aut}(A(X))_{c} {\text{ is finite}}\}.$
\end{thm}
\begin{proof}
Let $x$ be a point of $A_{X}.$
For any $y \in \partial A_{X}$,
we denote the hyperbolic line connecting $x$ to $y$ by $l_{x,y}.$

Let $P$ be the Dirichlet domain of $O^{+}(NS(X))$ centered at $x$.
Then there exists a horoball $U(c)$ based at $c$ such that

\begin{enumerate}
	\item $\{ \overline{U(c)} : c \in V_{NS}\}$ are mutually disjoint,
	\item $gU(c) = U(gc)$ for any element $g \in O^{+}(NS(X))$,
	\item $\{\overline{U(c)}\backslash {c} \} $ is locally finite in $\mathbb{H}^{n}_{X}$,
\end{enumerate}
where the set $V_{NS}$ is the same as in Definition 
\ref{the sets of bounded parabolic points}.
We put $K= \overline{P}\backslash \bigcup_{c\in C}U(c).$
Then, by Lemma \ref{generalized polytope2}, 
$\overline{A_{X}}^{H}\backslash \bigcup_{c\in C}U(c) = \operatorname{Aut}(A(X))\cdot K.$

If $y\notin C$,
since cusped regions $\{U(c): c\in V_{NS}\}$ are mutually disjoint,
there exists a sequence $\{a_{n}\}_{n\in \mathbb{N}}$ and 
$\gamma_{n} \in \operatorname{Aut}(A(X))$ 
such that  $a_{n}$ converges to infinity as $n\rightarrow \infty$ and 
$ l_{x,y}(a_{n}) \in \gamma_{n}K.$
Since the Euclidean diameter of $\gamma_{n}K$ goes to zero as $n$ goes to infinity,
we have 
$\displaystyle 
y
=\lim_{n\rightarrow \infty}l_{x,y}(a_{n})
=\lim_{n\rightarrow \infty}\gamma_{n}x \in \Lambda_{X}.$

If $y\in C$ and the rank of the stabilizer group $\operatorname{Aut}(A(X))_{y}$ is 0,
by Lemma \ref{horoballs}, the point $y$ is isolated in $\partial A_{X}.$
In particular, $y$ is not in the limit set $\Lambda_{X}$.

Otherwise, $y$ is in $C_{X}$,
where  $C_{X} = \{c \in C : \operatorname{vcd} W(X)_{c} < n-1 \}.$
Then $y$ is a bounded parabolic point. 
In particular, $y$ is in the limit set $\Lambda_{X}.$
\end{proof}
%%%%%%%%%%%%%%%%%%%%%%%%%%%%%%%%%%%%%%%%%%%%%%%%%%%%%%%%%%%%%%

\begin{defi}
Let $X$ be a K3 surface.
In Definition \ref{limit set of X}, 
we defined the limit set $\Lambda_{X}$ of $X$.
We denote the blown-up limit set of 
$\operatorname{Aut}_{s}(X)$ by $\Lambda_{X}^{bl}.$
\end{defi}

%%%%%%%%%%%%%%%%%%%%%%%%%%%%%%%%%%%%%%%%%%%%%%%%%%%%%%%%%%%%%
\begin{cor}\label{the blowing-up boundary of the ample cone}
If the symplectic automorphism group $\operatorname{Aut}_{s}(X)$
of a K3 surface $X$ is non-elementary, 
then $\partial A_{X}^{bl}\backslash \Lambda^{bl}_{X}= V\backslash C_{X}.$ 
\end{cor}
\begin{proof}
By Theorem \ref{Baragar's result for limit sets},
we have $\partial A_{X}\backslash V = \Lambda (\Gamma) \backslash C_{X}.$
Then, by Lemma \ref{the blowing-up boundary equals strict transform},
we have the desired equality.
\end{proof}

%%%%%%%%%%%%%%%%%%%%%%%%%%%%%%%%%%%%%%%%%%%%%%%%%%%%%%%%%%%%%%
Let $C$ be the set of cusped limit points of $\operatorname{Aut}(A(X)).$
By Lemma \ref{Ratcliffe_thm_12.6.5} and Lemma \ref{horopoints of Ax}, there exist cusped regions $U(c)$ based at $c \in C$
such that 
\begin{enumerate}
	\item $\{U(c) :c\in C\}$ are mutually disjoint,
	
	\item $gU(c) = U(gc)$ for any $g\in \operatorname{Aut}(A(X))$ and $c\in C,$
	
	\item $\{\overline{U(c)}\backslash \{c\}\}$ is locally finite in 
	$B^{n}\cup O(\operatorname{Aut}(A(X)))$.
	
	\item $U(c)$ meets just the sides of $\overline{A_{X}}$ incident with $c$.
\end{enumerate}
Since a cusped region $U(c)$ is homeomorphic to 
its strict transformation on $\overline{B}^{bl}$,
we identify $U(c)$ with its strict transformation.
To apply the Bestvina-Mess formula (Theorem \ref{Bestvina-Mess formula}) 
to automorphism groups of K3 surfaces,
we construct a homotopy on $\overline{A_{X}}^{bl}\backslash \cup_{c\in C}U(c)$.

%%%%%%%%%%%%%%%%%%%%%%%%%%%%%%%%%%%%%%%%%%%%%%%%%%%%%%%%%%%%%%%%
\begin{lem}\label{well-definess of homotopy}
Let $a$ be a point of $A_{X}$,
and let $K(a, t) = C(a, t) \cap \overline{A_{X}}\backslash \cup_{c\in C} U(c),$
where $C(a, t)$ is the closed $t$-ball centered at $a$.
For every $x \in \overline{A_{X}}\backslash (\cup_{c\in C} U(c)\cup C),$
let $S_{x}$ be the smallest (horo)sphere (based) centered at $x$ that meets $K(a, t)$.
Then $K(a, t) \cap S_{x}$ is a singleton.
\end{lem}
\begin{proof}
Suppose that there exists $y_{1}, y_{2}$ in $K(a, t)\cap S_{x}$ with $y_{1}\neq y_{2}.$
Then the hyperbolic line $\lbrack y_{1}, y_{2} \rbrack$ connecting $y_{1}$ to $y_{2}$
is contained in $C(a, t)\cap \overline{A_{X}}.$
For any $z\in (y_{1}, y_{2})$,
we have $d(x,z) < d(x, y_{1}) = d(x, y_{2}).$
Then we have $z \in \cup_{c\in C}U(c).$
Since cusped regions $U(c)$ are mutually disjoint,
there exists 
$c\in C$ such that 
$(y_{1}, y_{2}) \subset U(c)$ and $y_{1}, y_{2} \in K_{t}\cap \partial U(c).$

We now pass to the upper half-space model $U^{n}$ 
and conjugate $\operatorname{Aut}(A(X))$ so that $c=\infty.$
We denote by $\Sigma$ (resp. $w_{1}, w_{2}$) the horoball (resp. the point) on $U^n$
corresponding to $\partial U(c)$ (resp. $y_{1}, y_{2}$).
By Lemma \ref{horopoints of Ax},
the Euclidean segment $l$ connecting 
$y_{1}$ to $y_{2}$ is contained in $K(a,t)\cap \Sigma.$
This contradicts the minimality of $S_{x}$.
\end{proof}
%%%%%%%%%%%%%%%%%%%%%%%%%%%%%%%%%%%%%%%%%%%%%%%%%%%%%%%%%%%%%%
\begin{defi}
Let $K(a,t)$ be the same as in Lemma \ref{well-definess of homotopy}.
We put $$U = \cup_{c\in C} U(c)\cup C.$$

For every $x \in \overline{A_{X}}\backslash U$ and $t\geq 0$,
we define the function 
$$\rho_{a}:\lbrack 0, \infty ) \times 
\overline{A_{X}}\backslash U \rightarrow 
\overline{A_{X}}\backslash U$$ as follows:
if $x\in K(a, t)$, we define $\rho_{a}(x, t) = x,$
otherwise, we define $\rho_{a}(x, t)$ by the unique point of $K(a, t) \cap S_{x}$.
\end{defi}
%%%%%%%%%%%%%%%%%%%%%%%%%%%%%%%%%%%%%%%%%%%%%%%%%%%%%%%%%%%%%
\begin{prop}\label{continuity of homotopy}
Let $\rho_{a}$ be the same as above.
Then $\rho_{a}$ is continuous.
\end{prop}
\begin{proof}
If $\rho_{a}$ is not continuous, 
there exists a sequence $\{(x_{n}, t_{n})\}$ such that 
$(x_{n}, t_{n})$ converges to $(x, t)$ and
$\rho_{a}(x_{n}, t_{n})$ converges to $y\neq \rho_{a}(x, t).$

If $x$ is in $B^{n}$, passing to a subsequence of $\{(x_{n}, t_{n})\}$,
we can assume that each $x_{n}$ is in $B^{n}$.
By the definition of $\rho_{a}(x_{n}, t_{n})$ (resp. $\rho_{a}(x, t)$),
there exists $r_{n}$ (resp. $r$) such that 
$S(x_{n}, r_{n}) \cap K(a, t_{n}) = \{ \rho_{a}(x_{n}, t_{n}) \}$
(resp. $S(x, r) \cap K(a, t) = \{ \rho_{a}(x, t) \}$).
Passing to a subsequence of $(x_{n}, t_{n})$,
we can suppose that $r_{n}$ converges to $r_{0}$.
Since $d((x_{n}, t_{n}), \rho_{a}(x_{n}, t_{n}))= r_{n}$, 
we have $d(x, y) = r_{0}$. By the definition of $\rho_{a}(x, t)$,
we obtain $r_{0}>r$ and 
\begin{align*}
r_{n}- d((x_{n}, t_{n}), \rho_{a}(x,t)) 
&\geq r_{n}-d((x_{n}, t_{n}), (x, t)) -d((x, t), \rho_{a}(x,t)) \\
&= (r_{n}-r_{0}) +(r_{0}-r) -d((x_{n}, t_{n}), (x, t)).
\end{align*}
Hence, for a sufficiently large $n$,
we have $r_{n} > d((x_{n}, t_{n}), \rho_{a}(x,t))$.
This contradicts the definition of $r_{n}$.

If $x$ is in $\partial B^{n}$,
passing to the upper half-space model $U^n$,
there exists a (horo)sphere $S_{k}$ of $U^n$ (based) centered at $x_{k}$ such that 
$S_{k}\cap K(a, t_{k}) = \{\rho_{a}(t_{k},x_{k})\}.$
We denote the Euclidean center (resp. radius) of $S_{k}$ by 
$p_{k}=(p_{k,1}, \cdots, p_{k,n})$ (resp. $r_{k}$).
Then we have that $|p_{k}-\rho_{a}(t_{k}, x_{k})| = r_{k}$,
$|x_{k}-p_{k}|\leq r_{k}$, and 
$r_{k} \leq p_{k, n}$,
where $| \hspace{2mm}|$ is the Euclidean norm.
Passing to a subsequence of $\{(x_{n}, t_{n})\}$, we can assume that 
$p_{k}$ (resp. $r_{k}$) converges to some point $p=(p_{1}, \cdots, p_{n})$ (resp. $r\geq 0$).
As $k$ goes to infinity, we have 
$$|p-y| = r,\hspace{2mm}
|x-p| \leq r,\hspace{2mm}{\text{and }}
r \leq p_{n}.$$
Since $x$ is in $\partial B^{n}$, 
we have that $p=(x_{1}, \cdots, x_{n-1}, r)$ and 
$|x-p|=r$. Since $|p-y|=r$, this contradicts the uniqueness of $\rho_{a}(x,t)$.
\end{proof}

%%%%%%%%%%%%%%%%%%%%%%%%%%%%%%%%%%%%%%%%%%%%%%%%%%%%%%%%%%%%%
\begin{prop}\label{extends a homotopy}
$\rho_{a}$ extends to  $\overline{\rho_{a}}:\lbrack 0, \infty ) \times
\overline{A_{X}}^{bl} \backslash\cup_{c\in C} U(c)
\rightarrow
\overline{A_{X}}^{bl} \backslash\cup_{c\in C} U(c)
$.
\end{prop}
\begin{proof}
Let $x$ be a point in $\overline{A_{X}}^{bl} \backslash\cup_{c\in C} U(c)$ 
with $\pi_{C}(x)=c.$
We now pass to the upper half-space model $U^n$ so that $a=e_{n}$ and $c=\infty$
by an isometry $\phi: B^{n}\rightarrow U^{n}.$
Let $\Sigma$ be the horosphere $\{x\in U^n: x_{n}=r\}$
based at $\infty$ corresponding to $U(c).$

If $C(e_{n}, t)\cap \overline{A_{X}} \cap U(c) = \varnothing,$
we define $\overline{\rho_{a}}(x, t)$ 
by the intersection $S(e_{n}, t)$ with 
$\mathbb{R}_{>0}e_{n}.$
If $C(e_{n}, t)\cap \overline{A_{X}} \cap U(c) \neq \varnothing,$
there exists the half-line $l_{x}$ on $\Sigma$ starting at $re_{n}$ such that
$\operatorname{St}_{C}(\phi^{-1} \circ l)\cap \pi_{C}^{-1}(c) = \{x\}.$
For any $t\geq 0$, 
there exists a sufficiently large $R_{t}>0$
such that $\rho_{a}(s, x)$ is in $\Sigma$ 
for any $s\in (0, t\rbrack$ and any $x$ with $|x|>R_{t}$.
Then, if we define $\overline{\rho_{a}}(t, x)$ by the intersection of $K(a, t)$ with $l_{x},$ 
this is the desired extension of $\rho_{a}.$
\end{proof}

%%%%%%%%%%%%%%%%%%%%%%%%%%%%%%%%%%%%%%%%%%%%%%%%%%%%%%%%%%%%%
\begin{thm}\label{Main result}
If $\operatorname{Aut}_{s}(X)$
is non-elementary,
then $(\overline{A_{X}}^{bl}\backslash\cup_{c\in C} U(c), \partial A_{X}^{bl})$ is a $\mathcal{Z}$-structure
for $\operatorname{Aut}(X).$
In particular,
$$\operatorname{vcd}(\operatorname{Aut}(X)) = \dim \partial A_{X}^{bl} + 1.$$

\end{thm}
\begin{proof}
We define the homotopy $H:\lbrack 0, 1\rbrack \times
\overline{A_{X}}^{bl} \backslash\cup_{c\in C} U(c)
\rightarrow
\overline{A_{X}}^{bl} \backslash\cup_{c\in C} U(c)$ setting
$$ H(t, x) = 
\left\{
\begin{array}{ll} 
(t, x)  &  (t=1) \\
\rho_{a}(\log{\frac{t+1}{t-1}}, x) & (t\neq 1)
\end{array}
\right.$$
By Proposition \ref{extends a homotopy} and the definition of $\rho_{a}$,
$H$ is continuous.
Hence 
$\overline{A}_{X}^{bl}\backslash \cup_{c\in C} U(c)$ is an AR,
and 
$\partial A_{X}^{bl}$ is a $\mathcal{Z}$-set of $\overline{A}_{X}^{bl}\backslash \cup_{c\in C} U(c)$.
Then $(\overline{A}_{X}^{bl}\backslash \cup_{c\in C} U(c), \partial A_{X}^{bl})$
is a $\mathcal{Z}$-structure for $\operatorname{Aut}_{s}(X)$ (resp. $\operatorname{Aut}(X)$).
By Theorem \ref{Bestvina-Mess formula}, we have the desired equality.
\end{proof}

%%%%%%%%%%%%%%%%%%%%%%%%%%%%%%%%%%%%%%%%%%%%%%%%%%%%%%%%%%%

\begin{cor}\label{a cantor set}
If $\Lambda_{X}$ is homeomorphic to a Cantor set and $X$ has an elliptic fibration
whose rank is larger than $0$,
then $$\operatorname{vcd}(\operatorname{Aut}(X)) 
= \max \operatorname{MW}(f:X\rightarrow \mathbb{P}^{1}),$$
where $f$ runs all the elliptic fibrations of $X$.
\end{cor}
\begin{proof}
For any $c\in C$, we have
$\Lambda^{bl}_{X} \cap \pi_{C}^{-1}(c) \cong S^{r-1},$
where $r$ is the rank of $\operatorname{Aut}(A(X))_{c}$.
By Proposition \ref{P_vs_MW},
$r$ is equal to the rank of the Mordell weil group corresponding to $c$.
If $c \notin C$, since $\Lambda_{X}$ is totally disconnected, 
$\{c\}$ itself is an open neighborhood.
Then the topological dimension of $\Lambda^{bl}_{X}$ is equal to 
$\max \operatorname{MW}(f:X\rightarrow \mathbb{P}^{1})$.
\end{proof}

\begin{ex}
Let $X$ be the K3 surface whose N\'{e}ron-Severi lattice $NS(X)$ is given by 
$$NS(X) =
\begin{pmatrix}
2 & 4  & 1   \\
4 & 2  & 0   \\
1 & 0  & -2\\
\end{pmatrix}.
$$

Baragar \cite{Baragar3} studied the automorphism group of this K3 surface and the growth of orbits of curves.
He showed that $\Lambda_{X}$ is a Cantor set and $\operatorname{Aut}(A(X))$ is isomorphic to
$\mathbb{Z}/2\mathbb{Z} *\mathbb{Z}/2\mathbb{Z} *\mathbb{Z}/2\mathbb{Z}.$

Although the automorphism group of $X$ is completely determined in \cite{Baragar3} as above
and then this is not a new example of computing the virtual cohomological dimensions,
we can also apply Corollary \ref{a cantor set} to this K3 surface $X$.
\end{ex}

\subsection{The ample cones of K3 surfaces and sphere packings}
Let $X$ be a K3 surface,
and
let $\phi: \mathbb{H}_{X}^{n} \rightarrow B^n$
be an isometry from $\mathbb{H}_{X}^{n}$ to the conformal ball model $B^n$.
The isometry $\phi$ induces the natural homeomorphism from 
$\partial \mathbb{H}_{X}^{n}$
to $\partial B^n.$
For a $(-2)$-root $\delta$ and an ample class $a$ of $X$,
we set
\begin{align*}
H_{\delta}&=\{x \in \mathbb{H}_{X}^{n}: d(x,a)=d(x, s_{\delta}a)\}, \\
H_{\delta}^{+}&= \{x \in \mathbb{H}_{X}^{n}: d(x,a)<d(x, s_{\delta}a)\}.
\end{align*}
The closure of $H_{\delta}$ (resp. $H_{\delta}^{+}$) in $\overline{\mathbb{H}_{X}^n}$
is denoted by $\overline{H_{\delta}}$ (resp. $\overline{H}_{\delta}^{+}$).
We set 
$$\mathcal{S}_{X}=
\lbrace \delta \in \Delta^{+}(X): 
H_{\delta}\cap \overline{A_{X}}^{H} 
{\text{ is a side of }} \overline{A_{X}}^{H} \rbrace,$$
where $\overline{A_{X}}^{H}$ is the closure of $A_{X}$ in $\mathbb{H}_{X}^{n}.$

\begin{defi}
Let $\delta$ be an element of $\mathcal{S}_{X}$.
We define the {\it{sphere associated with $\delta$}} by 
$S_{\delta}= \phi(\overline{H_{\delta}})\cap \partial B^{n}$,
which is isomorphic to the $(n-2)$-sphere.
We define 
$B_{\delta}=\phi(\overline{H}_{\delta}^{+}\backslash \overline{H_{\delta}})\cap \partial B^n.$
We call $B_{\delta}$ the {\it{open ball associated with $\delta$}}.
\end{defi}

By Proposition \ref{ample cone vs Dirichlet domain}, 
we have 
$ A_{X} = \cap_{\delta \in \mathcal{S}_{X}} H_{\delta}^{+}$
and
$\partial A_{X} = \cap_{\delta \in \mathcal{S}_{X}} \overline{B_{\delta}},$
where $\overline{B_{\delta}}$ is the closure of $B_{\delta}$ in $\partial B^n.$

\begin{defi}
The boundary $\partial A_{X}$ associated with the ample cone of $X$ 
is a {\it{sphere packing}}
if for any two different $\delta_{1}$, $\delta_{2}$ in $\mathcal{S}_{X}$, their associated spheres $S_{\delta_{1}}$ and $S_{\delta_{2}}$
are either disjoint or tangent to each other.

$\partial A_{X}$ is called a {\it{connected sphere packing}}
if $\partial A_{X}$ is a sphere packing,
and for any two different $\delta$, $\delta^{'}$ in $\mathcal{S}_{X}$,
there exist $(-2)$-roots $\delta_{1}, \dots, \delta_{n}$ in $\mathcal{S}_{X}$ such that
$\delta=\delta_{1}$, $\delta^{'}=\delta_{n}$,
and two spheres $S_{\delta_{i}}$ and $S_{\delta_{i+1}}$ are tangent for $i=1, \dots, n-1.$
\end{defi}

\begin{thm}\label{Sphere packings}
Let $X$ be an elliptic K3 surface.
If $\partial A_{X}$ is a connected sphere packing,
and
every elliptic divisor of $X$ corresponds to 
the tangent point of some two spheres associated with elements of $\mathcal{S}_{X}$,
then
$$
\operatorname{vcd}(\operatorname{Aut}(X))
=\max \{ \operatorname{rk}\operatorname{MW}(f)\}
=\rho(X)-3
,$$
where $f$ runs all elliptic fibrations of $X$.
\end{thm}

%%%%%%%%%%%%%%%%%%%%%%%%%%%%%%%%%%%%%%%%%%%%%%%%%%%%%%%%%%%%
\begin{proof}
As $\partial A_{X}$ is connected,
we can sort 
the set of the open balls associated with elements of $\mathcal{S}_{X}$ to 
$\{ B_{i}\}_{i \in \mathbb{N}}$
such that 
for any $i\in \mathbb{N}$, 
the intersection of $B_{i}$ with $\cup_{j=1}^{i-1}B_{j}$
is not empty.
Let $\mathcal{P}_{m}$ be 
the set 
$\{p : p \in  B_{i} \cap B_{j} {\text{ for some two distinct }} i, j \leq m\}.$
We define the following sets:

\begin{itemize}
	\item $U_{m} = \bigcap_{i=1}^{m}B_{i}$,
	\vspace{2mm}
	\item $F_{m} = \bigcap_{i=1}^{m}\overline{B_{i}}$,
	\vspace{2mm}
	\item $F_{m}^{bl} =\operatorname{St}_{\mathcal{P}_{m}}(F_{m})
	=$ the strict transform of $F_{m}$ at $\mathcal{P}_{m}$ .		
\end{itemize}
For any $l<m$, let
$\pi_{\mathcal{P}_{l}, \mathcal{P}_{m}}$ be the same 
as in Definition \ref{blowing-up at bounded parabolic points}.
The map $\pi_{\mathcal{P}_{l}, \mathcal{P}_{m}}$ restricts to 
a map from $F_{m}^{bl}$ to $F_{l}^{bl}.$ These maps form a projective system.
Since $\partial A_{X}$ is a subset of $F_{m}$ for any $m\in \mathbb{N},$
there exists the natural map $\phi_{m}:\partial A_{X}^{bl} \rightarrow F_{m}^{bl}$.
Then we have the natural map $\phi : \partial A_{X}^{bl} \rightarrow \varprojlim F_{m}^{bl}$. 
The map $\phi$ is a homeomorphism.
To show this,
we consider the following:
$$
\psi_{m}:
\varprojlim F_{i}^{bl}
= \varprojlim \operatorname{St}_{\mathcal{P}_{i}}(F_{i}) \rightarrow
\cap_{i \in \mathbb{N}} \operatorname{St}_{\mathcal{P}_{m}}(F_{i}) =
\operatorname{St}_{\mathcal{P}_{m}}(\partial A_{X})
.$$
Since $\partial A_{X}^{bl}=\varprojlim \operatorname{St}_{\mathcal{P}_{m}}(\partial A_{X})$,
the maps $\psi_{m}$ induce 
the natural map $\psi : \varprojlim F_{m}^{bl}\rightarrow \partial A_{X}^{bl}$,
which is the inverse of $\phi.$

For each $m\in \mathbb{N}$,
$F_{m}^{bl}$ is a topological manifold with its interior $U_{m},$
and then $F_{m}^{bl}$ and $U_{m}$ are homotopy equivalent.
Hence we have 
\begin{equation}\label{the residual open sets}
H^{*}(\partial A_{X}^{bl}, \mathbb{Z})
\cong \lim_{m \to \infty} H^{*}(F_{m}^{bl},\mathbb{Z})
\cong \lim_{m \to \infty} H^{*}(U_{m},\mathbb{Z})
.
\end{equation}

For each $m \in \mathbb{N}$ and a sufficiently small $\epsilon>0$,
we define $V_{m+1}$ by 
the $\epsilon$-neighborhood of $\partial B^{n} \backslash B_{m+1}$ in $U_{m}.$
Then we can see as follows:
\begin{itemize}
	\item $V_{m+1}$ is contractible,
	\item $U_{m+1}\cup V_{m+1} = U_{m},$
	\item $U_{m+1}\cap V_{m+1}$ is homotopy equivalent to the $(n-2)$-sphere with finite $k$ punctures,
\end{itemize}
where $k$ is the number of the set 
$\{p: p {\text{ is a tangent point of }}
B_{i} {\text{ and }} B_{m+1} {\text{  for some  }}  i\leq m\}.$
By Mayer-Vietoris exact sequence 
and the $(n-2)$-sphere with $k$ punctures is
homotopy equivalent to the wedge sum of $(k-1)$-copies of $S^{n-3}$,
we have the following exact sequence:
$$
\cdots \rightarrow H^{*}(U_{m}) 
\rightarrow H^{*}(U_{m+1})\oplus H^{*}(V_{m+1}) 
\rightarrow H^{*}(\vee_{i=1}^{k-1} S^{n-3})
\rightarrow H^{*+1}(U_{m}) 
\rightarrow \cdots.$$
Therefore, we can inductively conclude that 
$$ H^{*}(U_{m}) = 0 \hspace{3mm}( *\geq n-2).$$
In particular, by Theorem \ref{Main result} and the equation \ref{the residual open sets},
we have 
\begin{align*}
 \operatorname{vcd}\operatorname{Aut}(X) 
 &= \max\{ m : H_{c}^{m}(A_{X}, \mathbb{Z})\neq 0\} \\
 &= \max\{ m : H^{m}(\partial A_{X}^{bl}, \mathbb{Z})\neq 0\}+1 \\
 &\leq n-2 =\rho(X)-3,
\end{align*}
where $H^{*}_{c}$ is the cohomology with compact support .
Lemma \ref{the rank of MW} implies that 
there exists an elliptic fibration $f$ such that 
$\operatorname{rk}\operatorname{MW}(f)= \rho(X)-3.$
\end{proof}

%%%%%%%%%%%%%%%%%%%%%%%%%%%%%%%%%%%%%%%%%%%%%%%%%%%%%%%%%%%%%%%%%%%%%%%%%%%
\begin{ex}\label{K3_the_Apollonial_circle_packing}
The Apollonian circle packing is generated as follows:
we begin with four mutually tangent circles.
There is four curvlinear triangles.
We can put the inscribed circle for each curvlinear triangle,
so that the curvlinear triangle separated into new curvlinear triangles.
Inductively, we put inscribed circles for each curvelinear triangle,
and the closure of the union of the inscribed circles is the Apollonian circle packing, 
which is shown in the following Figure.

\begin{figure}[H]
 \centering
 \includegraphics[width=5cm]{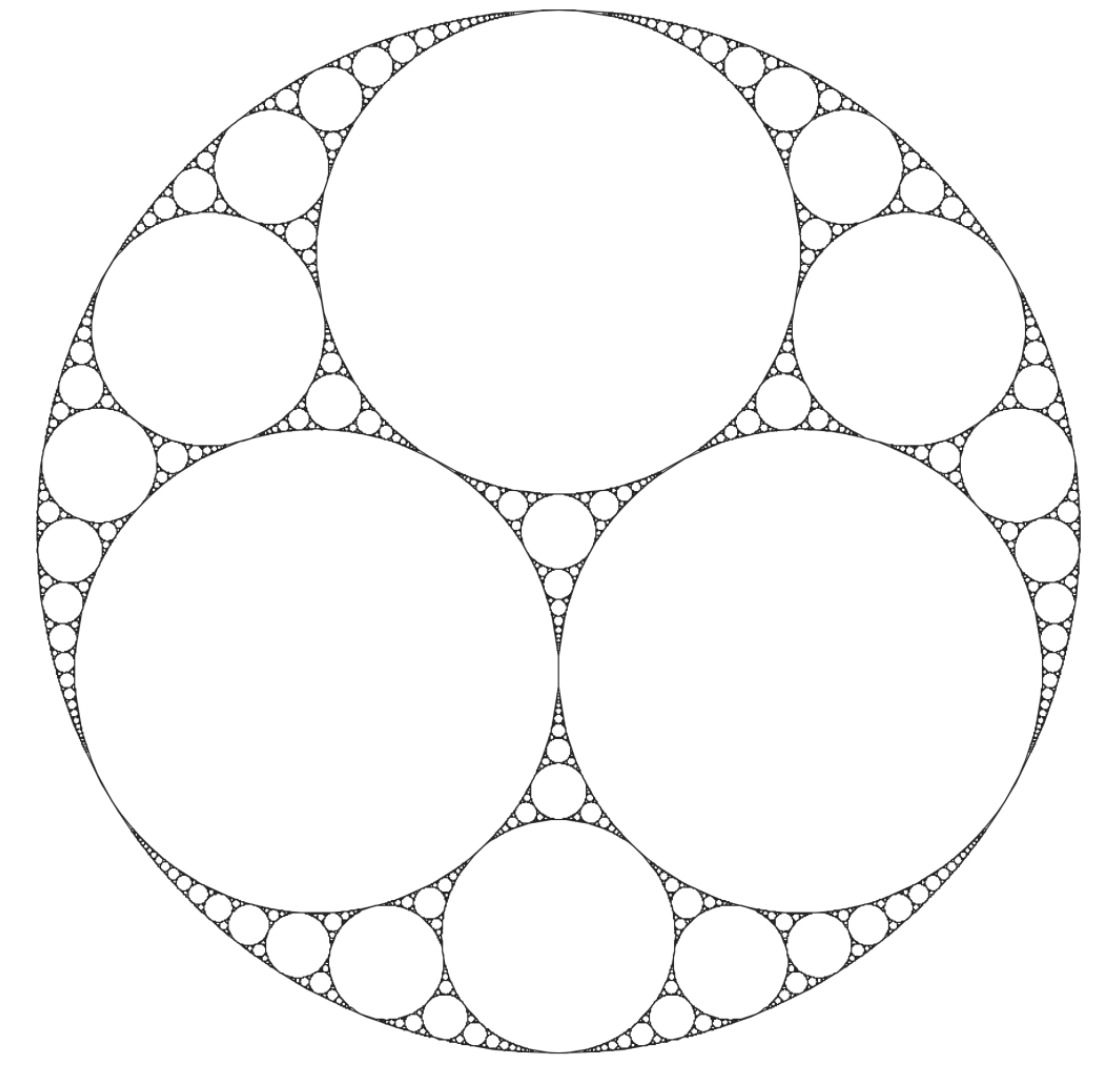}
 \caption{The Apollonian circle packing} \label{The_Apollonian_circle_packing}
\end{figure}

In \cite{Baragar1}, Baragar constructed a K3 surface $Y_{2}$ with ample cone 
the Apollonian circle packing, whose N\'{e}ron-Severi lattice is given by

$$
\begin{pmatrix}
-2 & 2  & 2  & 4  \\
2 &  -2  & 2  & 4 \\
2&  2  & -2 & 0  \\
4 & 4 & 0 &   0   \\
\end{pmatrix}.
$$
By Theorem \ref{Sphere packings}, 
we have $\operatorname{vcd} \operatorname{Aut}(Y_{2}) = 1.$ 
\end{ex}
%%%%%%%%%%%%%%%%%%%%%%%%%%%%%%%%%%%%%%%%%%%%%%%%%%
\begin{ex}\label{K3_the_Apollonial_sphere_packing}
The Apollonian sphere packing is the three dimensional analog of the Apollonian circle packing.
We begin with five mutually tangent spheres,
and in the space between any four of them,
we put the inscribed sphere.
As the same manner in the previous example,
we can inductively put a sphere so that this sphere touch other four spheres.
The closure of the union of the inscribed spheres is the Apollonian sphere packing.
Baragar \cite{Baragar1} also constructed a K3 surface $Y_{3}$
with ample cone the Apollonian sphere packing,
whose N\'{e}ron-Severi lattice is given by
 
$$
\begin{pmatrix}
-2 &  2  & 2  & 2 &  4  \\
2  & -2  & 2  & 2 &  4  \\
2  &  2  & -2 & 2 &  4  \\
2  &  2  & 2  &-2 &  0  \\
4  & 4   &  4 &0  &  0  \\
\end{pmatrix}.
$$
By Theorem \ref{Sphere packings}, 
we have $\operatorname{vcd}\operatorname{Aut}(Y_{3}) = 2.$ 
\end{ex}
%%%%%%%%%%%%%%%%%%%%%%%%%%%%%%%%%%%%%%%%%%%%%%%%%%%%%%%%%%%%%%%%%%%%%%%%%%%

%%%%%%%%%%%%%%%%%%%%%%%%%%%%%%%%%%%%%%%%%%%%%%%%%%%%%%%%%%%%%%%%%%%%%%%%%%%%
%%%%%%%%%%%%%%%%%%%%%%%%%%%%%%%%%%%%%%%%%%%%%%%%%%%%%%%%%%%%%%%%%%%%%%%%%%%%%%%%%%%%%%%%%%%%%%%%%%%%%%%


\begin{thebibliography}{99}
\bibitem{Alperin}
R. C. Alperlin, An elementary account of Selberg's lemma, 
L'Enseignement Math\'{e}matique {\textbf{33}} (1987), 269-273.
%\bibitem{Baragar1}
%A. Baragar, The Apollonian circle packing and ample cones for K3 surfaces.

\bibitem{Baragar1}
A. Baragar, The apollonian circle packing and ample cones for K3 surfaces, preprint.

\bibitem{Baragar2}
A. Baragar, The ample cone and orbits of curves on K3 surfaces, preprint.

\bibitem{Baragar3}
A. Baragar, Orbits of curves on certain K3 surfaces,
Compositio Mathematica {\textbf{137}} (2003), 115-134.

\bibitem{Bestvina-Mess}
M. Bestvina, G. Mess, The boundary of negatively curved groups, Journal of the american
mathematical society. Vol {\textbf{4}} (1991), 469-481.

\bibitem{Bestvina}
M. Bestvina, Local homology properties of boundaries of groups, Michigan Math. J. 43(1),123-139 (1996), 123-139.

\bibitem{BR}
D. Burns and M. Rapoport, On the Torelli problem for K\"{a}hlerian K3-surfaces, Ann. Sci. ENS. {\bf{8}} (1975), 235-273.

\bibitem{Brown}
K. S. Brown, Cohomology of Groups, (1982), Springer New York, NY.

\bibitem{Cossec-Dolgachev-Liedtke-Kondo}
F. Cossec, I. Dolgachev, C.Liedtke, S. Kondo,
Enriques surfaces I.

\bibitem{Fukaya}
T. Fukaya, Blown-up corona of relatively hyperbolic groups, preprint.

\bibitem{Fukaya-Oguni}
T. Fukaya and S. Oguni, Coronae of relatively hyperbolic groups and coarse cohomologies,
Journal of Topology and Analysis {\textbf{8}}(2016) 431-474.

\bibitem{Kondo1}
S.Kondo, The automorphism group of a generic Jacobian Kummer surface, J. Algebraic Geometry, \textbf{7}
(1998), 589-609.

\bibitem{Kondo2}
S. Kondo, K3 surfaces (Japanese), 2015, Kyoritu Shuppan, Japan. 

\bibitem{Ku1}
V. Kulikov, Degenerations of K3-surfaces and Enriques surfaces, Math. USSR., \textbf{11} (1977), 957-989.
\bibitem{Ku2}
V. Kulikov , Epimorphicity of the period mapping for K3 surfaces, Uspehi Mat. Nauk, \textbf{32} (1977), 257-258.

\bibitem{Mukai1}
S. Mukai, Finite groups of automorphisms of K3 surfaces and the Mathieu group, Invent. Math., \textbf{94}(1988), 183-221.

\bibitem{PS}
I. Piatetski-Shapiro and I. R. Shafarevich, A Torelli theorem for algebraic surfaces of type K3, Math. USSR. {\bf{5}} (1971), 547-587.

\bibitem{Ratcliffe}
J. G. Ratcliffe, Foundations of Hyperbolic manifolds (3rd edition), (2019), 
Springer Science \& Business Media.

\bibitem{Shioda2}
T. Shioda, On the Mordell-Weil lattices, Comment, Math. Univ. St. Paul, {\textbf{39}} (1990), 211-240.
\bibitem{Shioda}
T. Shioda and H. Inose, On Singular K$3$ Surfaces, Complex analysis and algebraic geometry,
Iwatani Shoten, Tokyo, 1977, 119-136.

\bibitem{Sterk}
H. Sterk, Finiteness results for algebraic K3 surfaces, Math. Z, {\textbf{189}} (1985), 507-513.

\end{thebibliography}
\end{document}